\documentclass[11pt]{elsarticle}
\usepackage[a4paper,margin=3cm]{geometry}
\usepackage{times}
\usepackage{amsthm}
\usepackage{amsmath}
\usepackage{amssymb}
\usepackage{mathtools}
\usepackage{relsize}%
\usepackage[T1]{fontenc}
\usepackage[utf8]{inputenc}
\usepackage{enumerate}
\usepackage{float}
\usepackage{verbatim}
\usepackage{hyperref}
\usepackage[autostyle, english = american]{csquotes}
\usepackage{paralist}

\usepackage{float}
\usepackage{subcaption}
\usepackage{tikz}
\usetikzlibrary{calc}
\usetikzlibrary{decorations,patterns}
\usepackage[framemethod=tikz]{mdframed}
\usepackage[active]{srcltx}

\theoremstyle{plain}
\newtheorem{theorem}{Theorem}[section]
\newtheorem*{theorem*}{Theorem}
\newtheorem{lemma}[theorem]{Lemma}
\newtheorem*{lemma*}{Lemma}
\newtheorem{corollary}[theorem]{Corollary}
\newtheorem*{corollary*}{Corollary}
\newtheorem{claim}[theorem]{Claim}
\newtheorem{remark}[theorem]{Remark\textbf{}}
\newtheorem*{remark*}{Remark}

\newtheorem{proposition}[theorem]{Proposition}

\theoremstyle{definition}
\newtheorem{definition}[theorem]{Definition}

\def\G{\mathcal{G}}

\def\mD{\mathbb{D}}

\def\deg{\mathrm{deg}}
\def\ds{\rule{0pt}{1.5ex}}

\newcommand{\ep}{\varepsilon}
\newcommand{\fal}{\lfloor \alpha \rfloor}
\newcommand{\ceal}{\lceil \alpha \rceil}

\newcommand\DOI[1]{{DOI:#1}}
\newcommand\junk[1]{}

\newcommand{\dd}[3]{\operatorname{\mathbb D}(#1,#2,#3)}

\newcommand{\dds}[4]{\operatorname{\mathbb D}(#1,#2,#3,#4)}

\newcommand{\lb}{\operatorname{LEG}}

\newcommand{\pppp}[1]{\partial^{++}(#1)}

\newcommand{\Smax}[1][r]{\Sigma_{\max}(#1)}
\newcommand{\amax}[1][r]{x_{\max}(#1)}

\newcommand{\Smin}[1][r]{\Sigma_{\min}(#1)}
\newcommand{\amin}[1][r]{x_{\min}(#1)}

\newcommand{\FG}{fully graphic}

\usepackage{xcolor}
\newcommand{\jms}{\varphi^*_{J\!M\!S}}

\newcommand{\hflip}[3]{(#1,#2)\Rightarrow (#1,#3)}

\newcommand{\appropto}{\mathrel{\vcenter{
  \offinterlineskip\halign{\hfil$##$\cr
    \propto\cr\noalign{\kern2pt}\sim\cr\noalign{\kern-2pt}}}}}

\begin{document}

\begin{frontmatter}


\title{Any fully graphic region of degree sequences can be sampled rapidly}

\author[renyi]{Péter L. Erdős\fnref{elp}}
\author[NEU]{Gábor Lippner\fnref{lg}}
\author[NEU]{Na'ama Nevo}
\author[renyi]{Lajos Soukup}
\address[renyi]{HUN-REN Alfréd Rényi Institute of Mathematics, Reáltanoda u 13--15 Budapest, 1053 Hungary\\
\texttt{$<$erdos.peter,soukup.lajos$>$@renyi.hun-ren.hu}}
\address[NEU]{Northeastern University, 360 Huntington Ave, Boston, MA, 02115\\ \texttt{$<$g.lippner,nevo.n$>$@northeastern.edu}}
\fntext[elp]{PLE was supported in part by NKFIH grant SNN~135643, K~132696 and he would like to thank the Institute for Computational and Experimental Research in Mathematics (ICERM) in Providence, Rhode Island, for its hospitality during the Fall research semester in 2024.}
\fntext[lg]{GL was supported in part by Simons Foundation Collaboration Grant No 953420}

\begin{abstract}
Let $n>c_1\ge c_2$ and $\Sigma$ be positive integers with $n\cdot c_1\ge \Sigma \ge n\cdot c_2.$ Let  $\mD=\dds{n}{\Sigma}{c_1}{c_2}$ denote the set of all degree sequences of length $n$ with the even sum $\Sigma$ and satisfying  $c_1\ge d_i\ge c_2.$ We show that if all degree sequences in $\mD$ are graphic, then $\mD$ is $3n^{13}$-stable. (The concept of $P$-stability  was introduced by Jerrum and Sinclair in 1990.) In particular, this implies that the switch Markov-chain mixes rapidly on all such degree sequences.

In this paper we also study the inverse direction. We show the following: if all graphic sequences of a degree sequence region satisfy the $p(n)$-stability condition then the overwhelming majority of the sequences in the region is graphic. This answers affirmatively a question raised in the paper  \DOI{10.1016/j.aam.2024.102805}.
\end{abstract}
\end{frontmatter}

\section{Introduction}\label{sec:intro}

Generating uniformly random graphs with a given degree sequence is a fundamental, and long researched, task in graph theory. There are two standard approaches to it: the configuration model and MCMC methods. On one hand, the configuration model has known limitations when the degree sequence is heterogeneous. On the other hand, while MCMC methods work well in practice, rigorous results on convergence rates are difficult to obtain.

The simplest MCMC method, proposed by Rao, Jana and Bandyopadhyay in \cite{RJB}, is the so-called swap Markov-chain: one step consists of replacing a pair of edges $(v_1w_1), (v_2w_2)$ with $(v_1 w_2), (v_2,w_1)$. This is allowed only when such a change does not produce multi-edges. It was conjectured by Kannan, Tetali, and Vempala in 1997~\cite{KTV97} that this Markov-chain mixes in polynomial time for all graphic degree sequences. However, the conjecture is open to this day. In fact, it has been verified only for a few families of degree sequences.

Jerrum and Sinclair~\cite{JS90} introduced a notion - called $P$-stability - for degree sequences, in order to study a different Markov-chain for graph generation. However, it turned out to be a useful concept for the swap Markov-chain as well. So useful, in fact, that in~\cite{P-stable} not only was it shown that $P$-stability implies fast mixing of the swap Markov-chain, but also that all previously known cases can be proved this way.

However, verifying $P$-stability of degree sequence families is still a challenging task. There are only a handful of known results in this direction - see Section~\ref{sec:old_results}.
\textbf{The main contribution of this paper is a novel method to show $P$-stability, and in turn polynomial-time mixing, for a broad new class of degree sequences.} Essentially, we show that when $P$-stability fails for a sequence, then there is another degree sequence with the same total number of edges and the same maximum and minimum degree that is not graphic. \textbf{We also show a partial converse to this:} if a degree sequence is not graphic, there is almost always an other (graphic) sequence with the same total number of edges and same maximum and minimum degree that is not $P$-stable. This positively answers a question raised in~\cite{fully}.

The outline of the paper is as follows. We introduce the key definitions in Section~\ref{sec:prelim}. In Section~\ref{sec:old_results} we list the most important known results about $P$-stability. Our new results are described in Section~\ref{sec:new_results}. The proofs of the main Theorems~\ref{tm:main} and~\ref{thm:sigmabound} are given in Sections~\ref{sec:fgps} and~\ref{sec:almost}, respectively.

\subsection{Preliminaries}\label{sec:prelim}

Here we recall (and slightly expand) the notions of various degree sequence regions from~\cite{fully} that are central to stating our results, as well as the concept of $P$-stability.

\paragraph{Degree sequence regions} A \textbf{degree sequence} is a sequence $\mathbf{d} = (d_1, \ldots , d_n)$ of positive integers, where no $d_i$ can exceed $n-1$, with even sum denoted by $\Sigma$. (In this paper, the parameter $\Sigma$ is always even.) We will often use $c_1 = c_1(\mathbf{d}) := \max (\mathbf{d})$ and $c_2 = c_2(\mathbf{d}) := \min (\mathbf{d})$) to denote the largest and smallest degrees of the sequence. We say $\mathbf{d}$ is \emph{graphical} if there exists a simple graph on the vertex set $[n]$ such that vertex $i$ has degree $d_i$ for every $i \in [n].$

We refer to a subset of all degree sequences given by a formula $\varphi = \varphi(\mathbf{d})$ as \textbf{degree sequence region} (or just region, for short) and denote it as
\begin{equation}
  \mD[\varphi] = \{ \mathbf{d} : \varphi(\mathbf{d})\}.
\end{equation}
We say that a region is \textbf{simple} if $\varphi$ depends only on $n, \Sigma, c_1, c_2$. In particular, for positive integers $n > a \geq b$ we denote
\begin{equation}\label{eq:VSR}
\dds{n}{\Sigma}{a}{b}=\mD\left[\ |\mathbf{d}|=n, a \geq c_1 \geq c_2 \geq b, \mbox{ and } {\textstyle\sum_{i=1}^n}d_i=\Sigma\ \right].
\end{equation}
When $\varphi$ doesn't even depend on $\Sigma$, we call the region \textbf{very simple}. In particular, we write
\begin{equation}\label{eq:SR}
\dd{n}{a}{b}= \mD[|\mathbf{d}|=n, a\geq c_1 \geq c_2 \geq b] = \bigcup \left\{\dds{n}{\Sigma}{a}{b}\ : \ n\cdot a\ge \Sigma \ge n\cdot b \right \}.
\end{equation}
Typically, a region $\mD$ may contain both graphic and non-graphic elements. We will be specifically interested in regions in which every sequence is graphic. Such regions are called \textbf{fully graphic}.

\paragraph{$P$-stability} The notion of $P$-stability of an infinite set of graphic degree sequences was introduced by Jerrum and Sinclair in \cite{JS90}. The notion plays a crucial role in complexity theory and graph theory.

Given a graphic degree sequence ${D}$ of length $|D|=n,$ let $\mathcal{G}(D)$ denote the set of all realizations of $D$ and let
\begin{equation}\label{eq:localP}
\pppp {{D}}=\sum_{1\le i< j\le n}{|\mathcal{G}({D}+1^{+i}_{+j})|}/{|\mathcal{G}({D})|},
\end{equation}
where the vector $1^{+i}_{+j}$  consists of all zeros, except at the $i$th and $j$th coordinates, where the values are $+1$. The operation $D\mapsto {D}+1^{+i}_{+j}$  is called  a \textbf{perturbation operation} on the degree sequences. We define the analogous operation $D+1^{-i}_{-j}$ similarly. Let us emphasize that we do not assume that $i$ and $j$ are different, so for example the operation $D \mapsto D+1^{+j}_{+j}$ is also  defined and it adds 2 to the $j^{th}$ coordinate.

An infinite region $\mD$ is called {\bf $P$-stable } if there is a polynomial $p(n)$ such that
\begin{equation}\label{eq:pstab}
\forall D \in \mD: \mbox{ $D$ is graphic} \Rightarrow \pppp{{D}}\le p(|{D}|)
\end{equation}
 The notion can be extended meaningfully for finite regions by fixing the polynomial $p(n)$ instead of just assuming its existence. Thus, we will say that a region is $p$-stable when it satisfies \eqref{eq:pstab} for a concrete polynomial $p = p(n)$. Let us emphasize that we do not require that   every element of a $P$-stable family to be graphic. $P$-stability has alternative, but equivalent, definitions  that use different perturbation operations (see the Appendix in \cite{fully} a for detailed explanation).

\subsection{Previous results}\label{sec:old_results}

Since the introduction of $P$-stability, only the following four infinite $P$-stable degree sequence regions were found:
\begin{enumerate}[(P1)]
\item (Jerrum, McKay, Sinclair \cite{JMS92}) The very simple  degree sequence region $\mathbb D[{\varphi}_{\ds \! {J\!M\!S}}]$ is $P$-stable, where $\varphi_{\ds \! {J\!M\!S}}\equiv (c_1-c_2+1)^2 \le 4 c_2(n-c_1-1)$.
\item (Jerrum, McKay, Sinclair \cite{JMS92}) The simple  degree sequence region $\mathbb D[\jms]$ is $P$-stable, where $    {\jms} \equiv (\Sigma-nc_2) (nc_1-\Sigma)\le (c_1-c_2)\left\{(\Sigma-nc_2)(n-c_1-1)+ (nc_1-\Sigma)c_2\right\}.$
\item (Greenhill, Sfragara \cite{G18}) The  simple degree sequence region $\mathbb D[{\varphi}_{\ds \! {GS}}]$ is $P$-stable, where  $\varphi_{\ds \! {GS}}\equiv (2\le c_2\text{ and }  3\le c_1 \le \sqrt{\Sigma/9})$. (This result was not announced explicitly, but \cite[Lemma 2.5]{G18} clearly proved this fact.)
\item (Gao, Greenhill \cite{GG20}) $\mathbb D[{\varphi}_{\ds \! {GG}}]$ is $P$-stable, where  $\varphi_{\ds \! {GG}}\equiv (\sum_{j=1}^{c_1} d_j + 6c_1+2 \le \sum_{j=c_1+1}^n d_j)$.  Here the degrees are understood to be in non-increasing order.
\end{enumerate}
Note that region (P1) is very simple, regions (P2) and (P3) are simple, while (P4) is not a simple region.

The paper~\cite{fully} was inspired by the observation that all of these regions are fully graphic. This is basically obvious for (P4) and was proved in~\cite{fully} for the others. This prompted a more thorough investigation of a potential connection between $P$-stability and the fully graphic property of regions. In~\cite{fully} it was shown that in the case of very simple regions the two notions are nearly equivalent. We summarize their main findings briefly, and without some of the technical details, in the following statements. We refer the reader to~\cite{fully} for the exact formulation of Theorem~\ref{tm:old2} below.

\begin{theorem}[{\cite[Theorems 4.2, 4.4]{fully}}]\label{tm:old}~
Every fully graphic very simple region is $(n^9)$-stable. Thus the maximal fully graphic very simple region  	
\begin{displaymath}
\mD_{\max}\coloneqq\bigcup\{\dd{n}{c_2}{c_1}:\text{ $\dd{n}{c_2}{c_1}$ is fully graphic}\}
\end{displaymath}
is $(n^9)$-stable, and the switch Markov chain is rapidly mixing on $\mD_{\max}$ with $O(n^{11})$ mixing time.
\end{theorem}

\begin{theorem}[{\cite[Theorem 5.8]{fully}}]\label{tm:old2}
If the very simple region $\dd{n}{c_1}{c_2}$ is not fully graphic then it is contained in a slightly larger very simple region $\dd{n}{c_1\cdot (1+\ep)}{c_2\cdot(1-\ep)}$ that is not $P$-stable, for some suitable $\ep = o(1)$ function.
\end{theorem}

\subsection{New results}\label{sec:new_results}

The goal of this paper is to extend the results of~\cite{fully} to simple regions. For the generalization of Part 1 of Theorem~\ref{tm:old} we show the following:
\begin{theorem}\label{tm:main}
If $\ \dds{n}{\Sigma}{c_1}{c_2}$ is fully graphic, then $\pppp{D}\le 3\cdot n^{13}$ for each  $D\in \dds{n}{\Sigma}{c_1}{c_2}$.
\end{theorem}

\begin{corollary}
The largest fully graphic simple region
\begin{displaymath}
\mathbb D\coloneqq\bigcup\left\{\mD(n,\Sigma, c_1,c_2):
\mD(n,\Sigma, c_1,c_2) \text{ is fully graphic} \right\}
\end{displaymath}
is $P$-stable.
\end{corollary}

To state the converse results, let us first recall from~\cite{fully} when a simple region $\dds{n}{\Sigma}{c_1}{c_2}$ is not fully graphic.

\begin{lemma}\label{lem:fully}
If $\dds{n}{\Sigma}{c_1}{c_2}$ is not fully graphic then
\begin{equation}\label{eq:fully-iff}
Q \coloneqq (c_1 - c_2 + 1)^2 - 4c_2(n - 1 - c_1) > 0
\end{equation}
and
\begin{equation}\label{eq:fully-iff2}
\left| \frac{\Sigma - n c_2}{(c_1 - c_2)/2} - (1 + c_1 + c_2) \right| \leq \sqrt{Q}+2.
\end{equation}
\end{lemma}

That is, for fixed $n,c_1,c_2$, the range of $\Sigma$ for which $\dds{n}{\Sigma}{c_1}{c_2}$ is not fully graphic is contained an interval of length roughly equal to $2\sqrt{Q}$. We will show, under mild extra conditions, that for most of these $\Sigma$ values the region is also not $P$-stable. More precisely, it contains a sequence $D$ with a large value of $\pppp D$.

\begin{theorem}\label{thm:sigmabound}
    Suppose $\beta > 0$,  $r\in \mathbb N$ and
    $\varepsilon = \tfrac{3(r+3)}{\beta(c_1 - c_2)}$.
    If
    \begin{equation}\label{eq:beta}
        (1-\beta)(c_1 - c_2 + 1)^2 \geq 4c_2(n - 1 - c_1).
    \end{equation}
    and
    \begin{equation}\label{eq:sigma_abs_bound}
    \left| \frac{\Sigma - n c_2}{(c_1 - c_2)/2} - (1 + c_1 + c_2) \right|
        \leq (1 - \varepsilon)\sqrt{Q},
    \end{equation}
    then
     the simple region $\mD(n,\Sigma,c_1,c_2)$ contains a sequence $D$ such that
    \begin{equation}\label{eq:notstable}
       \pppp D \ge  2^{r/2}.
        \end{equation}
        \end{theorem}

\begin{remark}
  Let us set $r = \log^2 n$. The $Q >0$ requirement already implies that $c_1 -c_2 \geq \Omega(\sqrt{n})$, hence the $\ep$ obtained is $o(1)$ as $n\to \infty$, as long as  $\beta > 0$ is fixed.  Thus, when the non-fully graphic $\Sigma$ range is not negligible, then almost all of these regions have a sequence $D$ for which \eqref{eq:notstable} implies
    \[\partial^{++}(D) \geq 2^{r/2} \geq n^{c\log n},
    \] so they are not $P$-stable.
\end{remark}

\section{Fully graphic simple regions are $P$-stable}\label{sec:fgps}

The main goal of this section is  to prove Theorem~\ref{tm:main}. This will be done in many steps. We begin by recalling alternating trails and their relevance for proving $P$-stability below. Section~\ref{sec:hostile} contains preliminary reductions, while Section~\ref{sec:initial} demonstrates the main method of our proof on a warm-up case. The actual proof is done in Sections~\ref{sec:refined} through~\ref{sec:caseII}. Finally, Section~\ref{sec:apply} contains an application of Theorem~\ref{tm:main}.

Let us fix the convention that the nodes of an $n$-vertex graph $G$ will be denoted by $V(G)=\{v_1,\dots, v_n\}$.

\begin{definition}\label{df:spec-trails}
    Let $G$ be a graph, and let $v$ and $v'$ be  (not necessarily distinct) vertices of $G$. An alternating trail of edges and non-edges between $v$ and $v'$ of  odd length is called \textbf{edge-deficient} if
    its first element  is a non-edge (hence the trail contains fewer edges than non-edges). Similarly,   it is called \textbf{edge-abundant} if
    its first element  is an edge, (hence  the trail contains more edges than non-edges). An edge-abundant trail of length at most $k$ is called \textbf{$k$-witness trail}. When the length is implicitly understood, we may simply call it a witness trail.
\end{definition}

The following well-known proposition shows why witness trails are relevant for comparing realizations of degree sequences.

\begin{proposition}\label{pr:find_trail}
    Assume that $D_0$ is a degree sequence of length $n$, and let $D_1 = D_0 + 1^{+p}_{+q}$.
    Suppose that the graphs $H_0$ and $H_1$ are realizations of $D_0$ and $D_1$, respectively.
    Then there exists a witness trail between $v_p$ and $v_q$ that alternates between the edges and non-edges of $H_1$.
    \end{proposition}

    \begin{proof}
    Observe that for each vertex $v \in V \setminus \{v_p, v_q\}$, the numbers of neighbors of $v$ in $E(H_0) \setminus E(H_1)$ and in $E(H_1) \setminus E(H_0)$ are the same.
    In contrast, both $v_p$ and $v_q$ have one more neighbor in $E(H_1) \setminus E(H_0)$ than in $E(H_0) \setminus E(H_1)$.

    Consider the symmetric difference $\Delta = E(H_0) \triangle E(H_1)$, and take a maximal trail in $\Delta$ such that
    \begin{inparaenum}[(i)]
    \item the trail alternates between edges from $E(H_1) \setminus E(H_0)$ and from $E(H_0) \setminus E(H_1)$,
    \item the first vertex of the trail is $v_q$, and
    \item the first edge of the trail belongs to $E(H_1) \setminus E(H_0)$.
    \end{inparaenum}
    Then the last vertex of the trail must be $v_p$, and the last edge must also belong to $E(H_1) \setminus E(H_0)$.
    Hence, this maximal trail satisfies the required conditions.
    \end{proof}

A key observation, already made in~\cite{JMS92}, is that the existence of short witness trails implies $P$-stability. This is also the first step in our approach.

\begin{theorem}\label{th:alter}
Assume that  $D\in \dds{n}{\Sigma}{c_1}{c_2}$, $1\le p,q\le n$ and  let $D^\diamond=D+ 1^{+p}_{+q}$.  Suppose  $G$ is a graph with
degree sequence $D^\diamond$, moreover,  $\Gamma_G(v_p)=\Gamma_G(v_q)$.
If $\dds{n}{\Sigma}{c_1}{c_2}$ is fully graphic, then  there exists an 11-witness trail between $v_p$ and $v_q$ in $G$.
\end{theorem}
\begin{proof}[Proof of Theorem \ref{tm:main} from Theorem \ref{th:alter}]
Assume that $G'$  is a graph whose degree sequence is $D'=D+1^{+i}_{+j}$ for some $1\le i\le j\le n$.

If $\Gamma_{G'}(v_i)=\Gamma_{G'}(v_j)$, then  we can apply Theorem \ref{th:alter} with $G=G'$, $p=i$ and $q=j$ to obtain  a  witness trail  $P$  between $v_i$ and $v_j$.  By flipping edges and non-edges along  $P$, we obtain a graph $G^\dag$ which is a realization of $D.$

If $\Gamma_{G'}(v_i)\ne \Gamma_{G'}(v_j)$, then  there exists an alternating trail $Q$ of length $2$ between $v_i$ and $v_j$. Assume that $Q=v_iv_mv_j$, where $(v_i,v_m)$ is an edge, and $(v_m,v_j$) is a non-edge. Flipping edges along   $Q$ yields a graph $G^*$ with degree sequence  $$D^*=D'+1^{+j}_{-i}=D+1^{+i}_{+j}+1^{-i}_{+j}=D+1^{+j}_{+j}.$$

Now, applying Theorem \ref{th:alter} to $G=G^*$ with $p=q=j$, we obtain  a witness trail $P$ from  $v_j$ to $v_j$.  Flipping edges and non-edges  $P$     results a graph $G^\dagger$, which is a realization of $D.$

To reconstruct $G'$ from $G^\dagger$, we need to determine whether we were  in case  where $\Gamma_{G'}(v_i)=\Gamma_{G'}(v_j)$ or in the case $\Gamma_{G'}(v_i)\ne \Gamma_{G'}(v_j)$.

If $\Gamma_{G'}(v_i)=\Gamma_{G'}(v_j)$, we should know $P$. The  trail  $P$ contains at most 12 vertices, so there are at most $n^{12}$ possibilities.

If $\Gamma_{G'}(v_i)\ne \Gamma_{G'}(v_j)$, we need to know both $P$ and $Q$. Since $P$ starts and ends at the same vertex, there are at most $n^{11}$ possibilities for $P$. Given $P$,  we can compute $G^*$ and identify either  $v_i$ or $v_j$. To determine   $Q$, we  need to know  its two further vertices, which gives $n^2$ possibilities for  $Q$. Knowing $Q$ we can compute  $G'$. We should know which endpoint  of $Q$ is $v_i$ and which is $v_j$. In this case we have at most $2\cdot n^2\cdot n^{11}=2\cdot  n^{13}$ possibilities.

Putting together, for a given $G^\dagger$ we have at most $n^{12}+2\cdot n^{13}\le 3\cdot n^{13}$ possibilities for $G'$, i.e. the operation $G'\mapsto G^\dagger$ from $\mathcal{G}({D}+1^{+i}_{+j})$ to  $\mathcal{G}({D})$ is
    ``$3\cdot n^{13}$-to-$1$''. Hence, $\pppp{D}\le 3\cdot n^{13}$.
\end{proof}

\subsection{Hostile configurations}\label{sec:hostile}

Before continuing our analysis, we introduce the notion of  a \textbf{hostile configuration},
which enables us to show that certain  degree sequences are not graphic.

\begin{definition}\label{df:hostile}
    Let $D'' \in \dds{n}{\Sigma}{c_1}{c_2}$,
    $1\le p\le q\le n$,   and let $G' \in \G(D''+1_{+q}^{+p}),$.
    We say that $G$ has a {\bf hostile configuration}
   iff  $V(G')$ can be partitioned into four subsets $V(G')=S \uplus K' \uplus Y' \uplus R'$
    satisfying the following properties:
    \begin{enumerate}[(a)]
    \item $G'[K']$  is a complete subgraph;
    \item $G'[K';R']$ is a complete bipartite subgraph;
    \item $G'[Y']$ is an empty subgraph;
    \item $G'[Y';R']$ is an empty bipartite subgraph;
    \item $S=\{v_p,v_q\}$ and   $\Gamma(v_p)\cup \Gamma(v_q) \subseteq K'$.
    \end{enumerate}
\end{definition}

\begin{lemma}\label{lem:Sou-3}
Let $D'' \in \dds{n}{\Sigma}{c_1}{c_2}$,
and $1\le p\le q\le n$. If   $G'\in \G(D''+1_{+q}^{+p}),$ is a hostile configuration, then  the degree sequence $D''$ is not graphic.
\end{lemma}
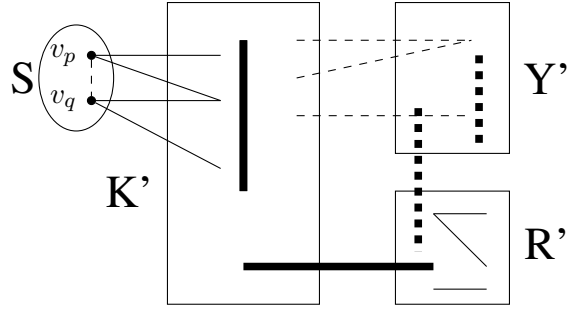
\begin{figure}[h]
\begin{center}
\begin{tikzpicture}[scale=1]
\tikzstyle{vertex}=[draw,circle,fill=black,minimum size=3,inner sep=0]
\draw  (-.2,5) ellipse (14pt and 20pt);
\node[vertex] (vp) at (0,5.3) [label=west:{$v_p$}] {};
\node[vertex] (vq) at (0,4.7) [label=west:{$v_q$}] {};
\node at  (-.9, 5) {\LARGE S};
\draw[dashed] (vp) -- (vq);
\draw (1,2) rectangle (3,6);
\node at (.5,3.5) {\LARGE K'};
\draw (vp) -- (1.7, 5.3);
\draw (vp) -- (1.7, 4.7);
\draw (vq) -- (1.7, 4.7);
\draw (vq) -- (1.7, 3.8);
\draw[line width=1mm] (2,5.5) -- (2, 3.5);
\draw (4,4) rectangle (5.5,6);
\node at (6,5) {\LARGE Y'};
\draw[line width=1mm] (2,5.5) -- (2, 3.5);
\draw[dashed] (2.7,5) -- (5,5.5);
\draw[dashed] (2.7,5.5) -- (5,5.5);
\draw[dashed] (2.7,4.5) -- (5,4.5);
\draw[line width=1mm, dashed] (5.1,5.3) -- (5.1,4.1);
\draw (4,2) rectangle (5.5,3.5);
\node at (6,2.8) {\LARGE R'};
\draw[line width=1mm, dashed] (4.3,4.6) -- (4.3,2.7);
\draw[line width=1mm] (2,2.5) -- (4.5, 2.5);
\draw (4.5,2.2) -- (5.2,2.2);
\draw (4.5,3.2) -- (5.2,3.2);
\draw (4.5,3.2) -- (5.2,2.5);
\end{tikzpicture}
\end{center}
\caption{The thick dashed lines indicate no edge at all, the solid thick lines indicate all edges.}	
\end{figure}
\begin{proof}
Assume on the contrary that $D''$ is graphic.
Then, by Proposition \ref{pr:find_trail}, there is an edge-abundant trail    $v_p=x_0,x_1,\ldots, x_{2m}, x_{2m+1}=v_q$ of odd length  that alternates between the edges and non-edges of $G'$.

Then, by induction on $0\le j\le 2m+1$, we can show that $x_j\in S\cup Y'$ whenever $j$ is even, and $x_j\in K'$ whenever $j$ is odd.  Indeed, if $x_{2i}\in S\cup Y'$ and $(x_{2i},x_{2i+1})$ is an edge, then $x_{2i+1}\in K'$ by (c), (d) and (e). Similarly, if $x_{2i+1}\in K'$ and $(x_{2i+1},x_{2i+2})$ is a non-edge, then $x_{2i+2}\in S\cup Y'$ by (a) and (b).

Hence, $x_{2m+1}\ne v_q$, which  a contradiction.
\end{proof}

The usefulness of hostile configurations becomes apparent from the following statement, which is at the heart of our argument.

\begin{lemma}\label{lem:key}
   Assume that  $D\in \dds{n}{\Sigma}{c_1}{c_2}$, $1\le p,q\le n$ and  let $D^\diamond=D+ 1^{+p}_{+q}$.  Let $G$ be a realization of
$D^\diamond$ where $\Gamma_G(v_p)=\Gamma_G(v_q)$. Suppose there is no 11-witness trail between $v_p$ and $v_q$ in $G$. Then there is a degree sequence $D''\in \dds{n}{\Sigma}{c_1}{c_2}$ for which $D'' + 1^{+p}_{+q}$ has a realization $G'$ having a hostile configuration.
\end{lemma}

Theorem~\ref{th:alter} is now a trivial consequence of Lemmas~\ref{lem:Sou-3} and~\ref{lem:key}. The rest of this section is devoted to the proof of Lemma~\ref{lem:key}.

\subsection{The initial structure and a warm-up case}\label{sec:initial}

From here on we use the notation and assumptions of Lemma~\ref{lem:key}. We start by recalling a result from~\cite{JMS92} showing that the graph $G$ under consideration has a certain structure. To do so, we define several pairwise disjoint subsets of the vertex set (introduced originally in \cite{JMS92}): Let
\begin{displaymath}
    S=\{v_p,v_q\},\qquad X= \Gamma(v_p) = \Gamma(v_q), \qquad Y=\{y\in V: |X\setminus \Gamma_G(y)|\ge 2\},
\end{displaymath}
furthermore let
\begin{displaymath}
    Z=\Gamma_G(Y)\setminus X,\qquad R=V\setminus (S\cup X\cup Y\cup Z),\qquad K=X\cup Z.
\end{displaymath}
The following statements were already explicitly formulated and proved  by Jerrum, McKay and Sinclair in \cite[Proof of Lemma 1]{JMS92} and were also used in
\cite[Proof of Lemma 4.5]{fully}.
\begin{lemma}\label{th:L1}
    Using the notation and terminology of Lemma~\ref{lem:key},
    assume the weaker condition that there is no $7$-witness trail between $v_p$ and $v_q$ in $G$
    (\emph{in Lemma~\ref{lem:key} we assume the stronger condition that there is no $11$-witness trail}).
    Then we have:
\begin{enumerate}[(i)]
\item $v_pv_q$ is not an edge;
\item $G[Y]$ is independent;
\item $G[K]$ is a clique;
\item $G[Y,R]$ is empty bipartite graph;
\item \label{en:Xr}$|X\setminus \Gamma_G(r)|\le 1$ for each $r\in R$.
\end{enumerate}
\end{lemma}

At this point we are in the position to give a short, calculation free proof for \cite[Lemma 4.5]{fully}. While this will not be directly used in the proof of Lemma~\ref{lem:key}, the method serves as an illustration of our approach.

\begin{lemma*}
Assume that  $D\in \dd{n}{c_1}{c_2}$, $1\le p,q\le n$ and  let $D^\diamond=D+ 1^{+p}_{+q}$.  Suppose  $G$ is a graph with vertex set $V=\{v_1,\dots v_n\}$ and
degree sequence $D^\diamond$, moreover,  $\Gamma_G(v_p)=\Gamma_G(v_q)$. If $\dd{n}{c_1}{c_2}$ is fully graphic, then  there exists a 7-witness trail between $v_p$ and $v_q$ in $G$.
\end{lemma*}

\begin{proof}
    The proof is contrapositive: we assume that there is no $7$-witness trail between $v_p$ and $v_q$ in $G$, and we will
    derive that $\dd{n}{c_1}{c_2}$ is not fully graphic.

    Fix $\Sigma$ with $D\in \dds{n}{\Sigma}{c_1}{c_2}$.
    Consider the sets $(S,X,Y,Z,R)$ defined in  Lemma \ref{th:L1}.
    Consider the graph $G'$ obtained from $G$ by removing all the edges from $R$. Let $D'$ be the degree sequence of $G'$.

    Observe that  $G'$ has a hostile configuration with $S=\{v_p,v_q\}$, $K'=X\cup Z$, $Y'=Y\cup R$ and $R'=\emptyset$. Thus $D'-1^{+p}_{+1}$ is not graphic.

    Since   $\deg_{G'}(v)=\deg_{G}(v)$   for $v\in V\setminus R$, and $\deg_{G'}(r)\ge |X|-1\ge n_2$ by \eqref{en:Xr}, we have  $D'-1^{+p}_{+p}\in \dd{n}{c_1}{c_2}$. Thus $\dd{n}{c_1}{c_2}$ is not fully graphic.
\end{proof}

Theorem~\ref{tm:old} now follows from the combination of this Lemma with Lemma~\ref{lem:Sou-3} in the same way as Theorem~\ref{tm:main} is obtained from Theorem~\ref{th:alter}.

    Although the proof of Lemma~\ref{lem:key} is substantially more intricate, its underlying structure closely parallels that of the above proof:
    First, assuming the non-existence of an $11$-witness trail, we refine the structural analysis by further partitioning the set $R$ into smaller subsets.
    Then, using hinge-flip operations, we modify the graph $G$ in such a way that
    (1) the degree sequence of the resulting graph remains in $\dds{n}{\Sigma}{c_1}{c_2}$, and
    (2) the final graph exhibits a hostile configuration.
    The core difficulty of the argument lies in the fact that it is not sufficient to remain within $\dd{n}{c_1}{c_2}$;
    we must ensure that all constructions stay within the more restrictive class $\dds{n}{\Sigma}{c_1}{c_2}$.

\subsection{Proof of Lemma~\ref{lem:key}: the refined structure}\label{sec:refined}

We will use the following notation: if $x,y,z,u,v,w$ form an alternating edge-deficient trail in $G$ starting with a non-edge, then we write:
\begin{displaymath}
x\cdots y - z \cdots u -v\cdots w.
\end{displaymath}
Any edge-abundant alternating trail can be described analogously.

\begin{definition}\label{df:testifier}
    An edge-deficient trail with length at most five is called \textbf{pre-witness trail}.
\end{definition}
\begin{lemma}\label{lm:sou1-2}
(1)    There is no pre-witness trail in $G[K\cup R]$ between two different vertices of $K$. \\
\noindent (2) There is no pre-witness trail in $G[K\cup R]$ from a vertex
$w\in K$ to itself provided $|\Gamma_G(w)\cap Y|\ge 2.$
\end{lemma}

\begin{proof}
(1) Assume on the contrary that  $w_1\cdots a_1-a_2\cdots a_3-a_4 \cdots w_2$ is a pre-witness trail in $G[K\cup R]$ between two different points of $K$. Pick $y_i\in {Y\cap \Gamma_G(w_i)}$ for $i=1,2$. Since $|X\setminus \Gamma_G(y_i)|\ge 2$ we can pick $\{x_1,x_2\}\in {[X]}^{2}$ such that $(x_i,y_i)$  is non-edge for $i=1,2$. Then $v_p-x_1\cdots y_1-w_1\cdots a_1-a_2\cdots a_3-a_4 \cdots w_2-y_2\cdots x_2-v_q$ is a witness trail because there is no edge repetition in it because         $w_1\ne w_2$ and $x_1\ne x_2$.

\noindent (2). Assume on the contrary that  $w\cdots a_1-a_2\cdots a_3-a_4 \cdots w$ is a pre-witness trail in $G[K\cup R]$ with $w\in K$. Pick $\{y_1,y_2\}\in {[Y\cap \Gamma_G(w)]}^{2}$. Pick $\{x_1,x_2\}\in {[X]}^{2}$ such that $(x_i,y_i)$ for $i=1,2$ are non-edges. Then $v_p-x_1\cdots y_1-w\cdots a_1-a_2\cdots a_3-a_4 \cdots w-y_2\cdots x_2-v_q$ is a  witness trail because there is no edge repetition in it because $y_1\ne y_2$ and $x_1\ne x_2$.
\end{proof}

\medskip\noindent Recall that $K=X\cup Z=\{v_1,\dots, v_k\}$. Consider the following partition of $R$:
\begin{enumerate}[(r1)]
\item $R_0=\{r\in R: K\subset  \Gamma_H(r)\}$,
\item $R_i=\{r\in R: K\setminus\Gamma_H(r)=\{v_i\}\}$ for $1\le i\le k$,
\item $R_\infty=\{r\in R: |K\setminus \Gamma_H(r)|\ge 2\}$.
\end{enumerate}
(See Figure \ref{figure2}.)

\begin{figure}[H]
\begin{center}
\begin{tikzpicture}[scale=0.8]
\tikzstyle{vertex}=[draw,circle,fill=black,minimum size=3,inner sep=0]
\draw (0,6) rectangle (2,9);
\node at (-.5,7.5) {X};
\node[vertex] (v1) at (1,8.5) [label=east:{$v_1$}] {};
\draw (3,6) rectangle (5,8);
\node[vertex] (vj) at (1,3.2) [label=east:{$v_j$}] {};
\node[vertex] (vk) at (1,.2) [label=east:{$v_k$}] {};
\node[vertex] (vjj) at (1,2.8)  {};
\node[vertex] (vkk) at (1,.6)  {};
\node at  (5.5, 7) {Y};
\node[vertex] (y) at (4,7.5) [label=east:{$y$}] {};
\draw[dashed] (y) -- (1.5,7.5) ;
\draw[dashed] (y) -- (1.5,7) ;
\node[vertex] (y') at (3.5,6.5) [label=east:{$y'$}] {};
\node[vertex] (vp) at (3.5,9) [label=east:{$v_p$}] {};
\node[vertex] (vq) at (3.5,8.4) [label=east:{$v_q$}] {};
\draw  (3.5,8.7) ellipse (20pt and 14pt);
\draw[dashed] (vp) -- (vq);
\node at  (4.5, 8.5) {S};
\draw (0,0) rectangle (2,5);
\node at (-.5,2.5) {Z};
\draw (3,4.6) rectangle (4,5); \node at  (4.5, 4.8) {$R_0$};
\draw (3,4) rectangle (4,4.4); \node  at  (4.5, 4.2) {$R_1$};
\draw[dashed] (3.2,4.2) .. controls (1,7).. (v1) ;
\draw (3,3) rectangle (4,3.4); \node at  (4.5, 3.2) {$R_j$};
\draw[dashed] (3.5,3.2) .. controls (2.2,3.7).. (vj) ;
\draw (3,0) rectangle (4,.4); \node at  (4.5, .2) {$R_k$};
\draw[dashed] (3.8,0.2) .. controls (2.2,-.7).. (vk) ;
\draw (y') -- (1,4.6) ;
\draw (y') .. controls (0,3).. (.5,2) ;
\draw (6,.5) rectangle (7.5,4.5); \node at  (8, 2.5) {$R_\infty$};
\coordinate (xR) at (6.5,2);
\node[vertex] at (xR) {};
\draw[dashed] (xR) -- (vjj);
\draw[dashed] (xR) -- (vkk);
\end{tikzpicture}
\end{center}
\caption{The dashed lines indicate non-edges, the solid lines indicate edges.}	\label{figure2}
\end{figure}
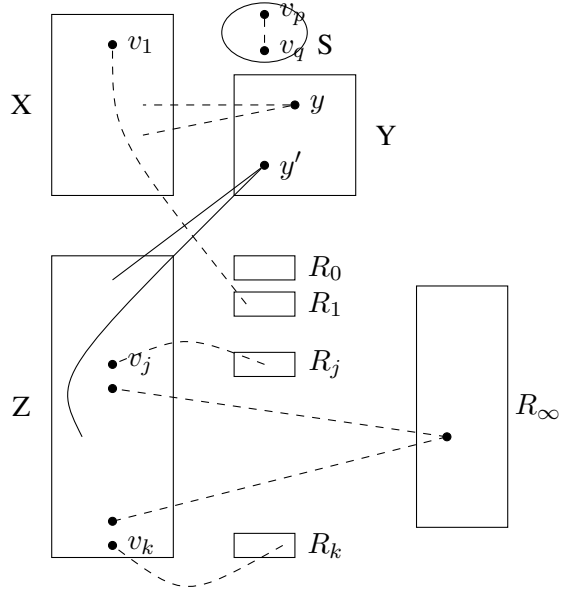

\begin{lemma}\label{lm:Rinf-free}
$R_\infty$ does not contain edges.
\end{lemma}
\begin{proof}
Assume on the contrary that  $(x,y)$ in an edge in $R_\infty.$  Since
$|K\setminus \Gamma_G(x)|\ge 2$    and  $|K\setminus \Gamma_G(y)|\ge 2$, we can choose two distinct vertices,  $v_i\ne v_j$ from  $K$ such that $(x,v_i)$ and $(y,v_j)$ are non-edges. Then $v_i\cdots x-y\cdots v_j$ is a forbidden pre-witness trail of length $3$. Contradiction.
\end{proof}

\begin{lemma}\label{lm:RiRj_free}
There is no edge between $R_i$ and $R_j$ for $i\ne j \in (\{1,2,\dots,k \}\cup\{\infty\})$
\end{lemma}
\begin{proof}
Assume on the contrary that $a\in R_i$, $b\in R_j$ and $(a,b)$ is an edge.
We can assume that $i\ne \infty$.
Then $(a,v_i)$ is a non-edge by the definition of $R_i$.

If $j=\infty $ pick $k\ne i$ such that $(b,v_k)$ is a non-edge. If, however, $j<\infty$, then let $k=j$. In both cases $k\ne i$ and $(v_k,b)$ is a non-edge. Hence, $v_i\cdots a-b\cdots v_j$ is a forbidden pre-witness trail of  length $3$.
\end{proof}

\begin{lemma}\label{lm:Ysingleton}
Assume  $(a,b)$ is an edge in $R_i$   for some $1\le i\le k$. Then $|\Gamma_G(v_i)\cap Y|=1$.
\end{lemma}

\begin{proof}
By Lemma \ref{lm:sou1-2}(2), we have  $|\Gamma_G(v_i)\cap Y|\le 1$ because $x_i\cdots a-b\cdots x_i$ is a pre-witness trail.
\end{proof}

\begin{lemma}\label{lm:Ri-empty}
If there exists an edge $(a,b)$ in $R_i$  for some  $1\le i\le k$ then  $R_j=\emptyset$  provided  $1\le j\le k$ and $ j\ne i$.
\end{lemma}

\begin{proof}
Assume on the contrary that there exists $c\in R_j$ for some  $j\ne i, 1\le j \le k.$ Observe that $(x_i,c)$ is an edge because $c\in R_j$ and $j\ne i$. Hence,  we have a pre-witness trail between $x_i$ and $x_j$ of length $5$: $x_i\cdots a-b\cdots x_i-c\cdots x_j$.  (see Fig. \ref{fig}). Contradiction.
\end{proof}

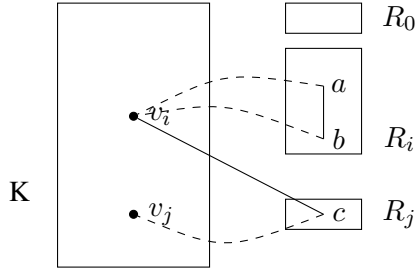
\begin{figure}[h]
\begin{center}
\begin{tikzpicture}[scale=1]
\tikzstyle{vertex}=[draw,circle,fill=black,minimum size=3,inner sep=0]
\node[vertex] (vi) at (1,3.5) [label=east:{$v_i$}] {};
\node[vertex] (vj) at (1,2.2) [label=east:{$v_j$}] {};
\draw (0,1.5) rectangle (2,5);
\node at (-.5,2.5) {K};
\draw (3,4.6) rectangle (4,5); \node at  (4.5, 4.8) {$R_0$};
\draw (3,3) rectangle (4,4.4); \node at  (4.5, 3.2) {$R_i$};
\draw (3.5,3.2) -- (3.5,3.9);
\draw[dashed] (3.5,3.2) .. controls (2.2,3.7).. (vi) ;
\draw[dashed] (3.5,3.9) .. controls (2.2,4.1).. (vi) ;
\draw (3,2) rectangle (4,2.4); \node at  (4.5, 2.2) {$R_j$};
\draw[dashed] (3.5,2.2) .. controls (2.2,1.7).. (vj) ;
\draw (vi) -- (3.5,2.2);
\node at (3.7,3.9) {$a$};
\node at (3.7,3.2) {$b$};
\node at (3.7,2.2) {$c$};
\end{tikzpicture}
\end{center}
\caption{We provide the complete witness trail.}\label{fig}
\end{figure}
\medskip\noindent
After this preparation we can see that
we can distinguish two cases concerning the structure of $G$.

\begin{description}
\item [Case I: ]  $R_N\coloneqq\bigcup\{R_i: 1\le i \le k\}\cup R_\infty$ is an  independent  subset,
\item [Case II:] There exists  $1\le i \le k$ such that  $R_i$ contains edges, and so $R=R_0\cup R_i\cup R_{\infty}$.
\end{description}

\medskip\noindent
Our goal in the remaining part of the proof is to transform our degree sequence $D$ into a new degree sequence $D''\in \dds{n}{\Sigma}{c_1}{c_2}$ and our graph $G$ into a realization $G' \in \G(D''+1_{+q}^{+p})$ which is a hostile configuration.

\begin{definition}\label{df:hingeflip}
Assume that $G$ is a graph with degree sequence $D$, and let $x,y,z$ be distinct vertices such that $(x,y)$ is an edge and $(x,z)$ is a non-edge. The \textbf{hinge-flip} operation $\hflip{x}{y}{z}$ deletes the edge $(x,y)$ from $H$ and adds the edge $(x,z)$ to obtain a new graph $H'$. The  degree sequence of $H'$ is $D'=D+1^{-i}_{+j}.$ The operation was already introduced in \cite{JS89} as \textbf{Type 0 transition}.
\end{definition}
\noindent The hinge-flip operation clearly changes the degree sequence of $H$, since $\deg_{H'}(y)=\deg_H(y)-1$ and $\deg_{H'}(z)=\deg_H(z)+1.$ The other degrees are not  changed.

\subsection{Proof of Lemma~\ref{lem:key}: Case I. There is no edge in $R_N$.}\label{sec:caseI}

We generate a sequence of graphs $G_0, G_1,\ldots, G_m$ and their degree sequences   $D_0, D_1, \ldots, D_m$ with consecutive hinge-flip operations as follows.

Let $G_0=G$, and so $D_0=D^\diamond=D+1^{+p}_{+q}$. Assume that $G_\ell$ is already produced. Pick $x_\ell \ne z_\ell \in R_0$ and $y_\ell \in R_N$ such that $(x_\ell, y_\ell)$ is an edge and $(x_\ell, z_\ell)$ is a non-edge in $G_\ell.$ Apply the corresponding hinge-flip operation
$\hflip{x_\ell}{y_\ell}{z_\ell}$ for $G_\ell$ to obtain $G_{\ell+1}$. (We will call this operation a \textbf{downward twist}.) We stop the process if there are no suitable $x_\ell$, $y_\ell$ and $z_\ell$.

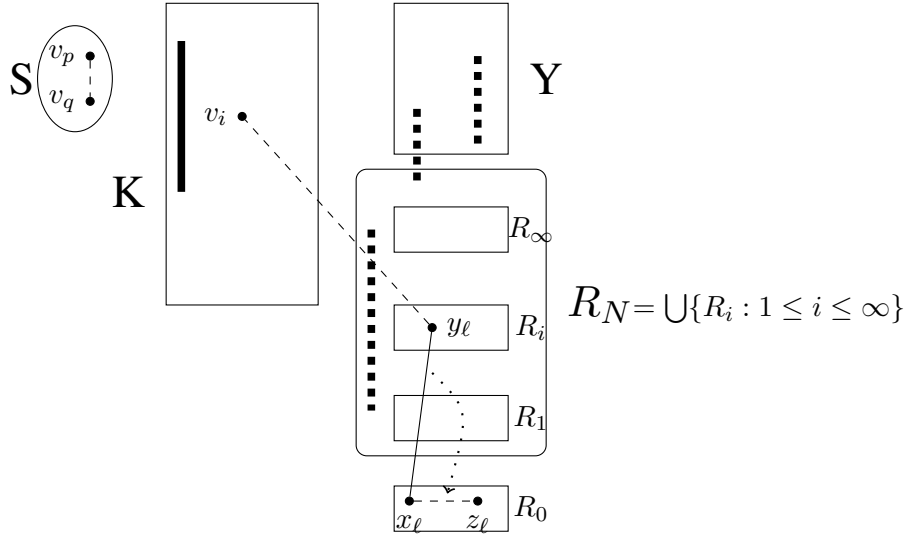
\begin{figure}[h]
\begin{center}
\begin{tikzpicture}[scale=1]
\tikzstyle{vertex}=[draw,circle,fill=black,minimum size=3,inner sep=0]
\draw  (-.2,5) ellipse (14pt and 20pt);
\node[vertex] (vp) at (0,5.3) [label=west:{$v_p$}] {};
\node[vertex] (vq) at (0,4.7) [label=west:{$v_q$}] {};
\node at  (-.9, 5) {\LARGE S};
\draw[dashed] (vp) -- (vq);
\draw (1,2) rectangle (3,6);
\node at (.5,3.5) {\LARGE K};
\node[vertex] (vi) at (2,4.5) [label=west:{$v_i$}] {};
\draw (4,4) rectangle (5.5,6);
\node at (6,5) {\LARGE Y};
\draw[line width=1mm] (1.2,5.5) -- (1.2, 3.5);
\draw[line width=1mm, dashed] (5.1,5.3) -- (5.1,4.1);
\draw[rounded corners] (3.5,0) rectangle (6,3.8);
\draw[line width=1mm, dashed] (4.3,4.6) -- (4.3,3.6);
\draw (4,2.7) rectangle (5.5,3.3);
\node at (5.8,3) {$R_\infty$};
\draw (4,.2) rectangle (5.5,0.8);
\node at (5.8,.5) {$R_1$};
\node at (8.5,2) {\scalebox{1.5}{$R_N$}$=\bigcup\{R_i : 1\le i \le \infty\}$};
\draw[line width=1mm, dashed] (3.7,3) -- (3.7,.6);
\draw (4,1.4) rectangle (5.5,2);
\node at (5.8,1.7) {$R_i$};
\draw[dashed] ;
\draw (4,-1) rectangle (5.5,-.4);
\node at (5.8,-.7) {$R_0$};
\node[vertex] (b) at (4.5,1.7) [label=east:{$y_\ell$}] {};
\node[vertex] (a) at (4.2,-.6) [label=south:{$x_\ell$}] {};
\node[vertex] (c) at (5.1,-.6) [label=south:{$z_\ell$}] {};
\draw[dashed] (a) -- (c);
\draw (a) -- (b);
\draw[dashed] (b) -- (vi);
\draw[loosely dotted,thick,->] (4.5,1.1) .. controls (5,.6) .. (4.7,-.5);
\end{tikzpicture}
\end{center}
\caption{In $R_N$ there is no edge. The vertices $x_j, z_j\in R_0$ and $y_j\in R_N$. The curved, loosely dotted arrow indicates a hinge-flip operation. As earlier: the thick dashed lines indicate no edge at all, the solid thick lines indicate all edges.}	
\end{figure}

\noindent

\begin{lemma}\label{lm:hinge1}
$D_m+1^{-p}_{-q}\in \dds{n}{\Sigma}{c_1}{c_2}$.
\end{lemma}
\begin{proof}
The hinge-flip operations  keep the sum of the degrees, so  $\Sigma(D_{m}) =  \Sigma(D_{0}) =\Sigma(D)+2$. Since we added and removed edges only inside $R_N\cup R_0$, we  have to show that $$c_1 \ge \deg_{G_m}(v) \ge c_2$$ whenever $v\in R_N\cup  R_0.$ Since $\Gamma_{G_m}(v)\cap X=\Gamma_{G}(v)\cap X$ and $|X\setminus \Gamma_{G}(v)|\le 1$ by \eqref{en:Xr}, we have $\deg_{G_m}(v)\ge |X|-1\ge c_2$.

For $v\in R_N$, the applied hinge-flip operations can not increase the degree in $v$. So $ \deg_{G_m}(v)\le \deg_G(v)\le c_1$.
 So, to complete the proof of the Lemma, we should show $\deg_{G_{\ell+1}}(v)\le c_1$ for each $v\in R_0$. Assume on the contrary that $\deg_{G_{m}}(z)>c_1$ for some $z\in R_0$. Since $\deg_{G_0}(z)\le c_1$, we increased the degree of $z$  by some   hinge-flip operation $\hflip{x_\ell}{y_\ell}{z_\ell}$, and so $z=z_\ell.$

Since $y_\ell\in R_N$, we can pick  $1\le i\ne j  \le k$  such that $v_iy_\ell$ is an edge and $v_jy_\ell$ is a non-edge in
$G_\ell$.   Moreover, the hinge-flip operations
$\hflip{x_s}{y_s}{z_s}$
modify the edges only inside $R$, so
\begin{center}
    $v_iy_\ell$ is an edge and $v_jy_\ell$ is a non-edge in $G$
\end{center}
as well. We know that the bipartite graph $G_\ell[R_0,\{v_j\}]$ is complete. Hence, we have the following inequality chain:
\begin{displaymath}
\deg_{G\setminus R_0}(z_\ell)+|R_0|\ge \deg_{G_m}(z_\ell)>c_1\ge \deg_G(v_j)= \deg_{G\setminus R_0}(v_j)+|R_0|.
\end{displaymath}
Thus,  there exists a vertex $d$ such that $(z_\ell,d)$ is an edge, and
$(v_j,d)$ is a non-edge in $G$. Then $d\notin Y\cup S$ because $(z_\ell,d)$ is an edge and $z_\ell \in R_0$. Moreover,  $d\notin K\cup R_0$ because $v_jd$ is a non-edge. Thus, the only possibility is that  $d\in R_N$.

Hence, $v_j\cdots d-z_\ell\cdots x_\ell- y_\ell\cdots  v_i$ is a forbidden pre-witness trail
of length $5$
which does not
exist by   Lemma \ref{lm:sou1-2}(1).
\end{proof}
\noindent Define
\begin{displaymath}
    R_0^0:=\left\{r\in R_0 : \Gamma_m(r) \cap R_N \ \text{is not empty}\right\}\quad \text{and}\quad R_0^1=R_0\setminus R_0^0.
\end{displaymath}
\begin{lemma}\label{lm:R00}
$R_0 \subseteq \Gamma_m(r)$ for each $r\in R^0_0$.
\end{lemma}

\begin{proof}
Assume on the contrary that $z\in R_0$ and $(z,r)$ is a non-edge in $G_m$.
Since $r\in R^0_0$, there is $y\in R_N$ such that $(r,y)$ is an edge in $G_m$. So we have a  further  downward twist operation  $\hflip{r}{y}{z}$ in $G_m$. Contradiction.
\end{proof}

\begin{lemma}\label{lm:dream-I}
    $G_m$ has  a  hostile configuration.
\end{lemma}
\begin{proof}
We have $S=\{v_p,v_q\}$. Define
\begin{displaymath}
K':= K\cup R_0^0, \quad Y':=Y\cup R_N, \quad R':=R_0^1.
\end{displaymath}
 We should check that all conditions in Definition \ref{df:hostile}. (a)-(e)  hold.

\begin{enumerate}[(a)]
\item {\bf $G_m[K']$  is a complete subgraph.} \quad Indeed, $K$ is a clique by Lemma \ref{th:L1}, $[K,R_0]$ is complete by  the definition of $R_0$, and  $R^0_0$ is a clique by Lemma \ref{lm:R00}.

\item {\bf $G_m[K';R']$ is a complete bipartite subgraph.} \quad
Indeed, $R'=R^1_0\subset R_0$ and $[K,R_0]$ is complete by the definition of $R_0$, so $[R',K]$ is complete. Moreover, $[R^0_0,R^1_0]$ is complete by Lemma \ref{lm:R00}.

\item {\bf $G_m[Y']$ is an empty subgraph.} \quad Indeed, $Y$ is independent by Lemma \ref{th:L1}. $[R,Y]$ is empty by Lemma \ref{th:L1}. $R_N$ is independent because we are in \textbf{Case I}.

\item {\bf $G_m[Y';R']$ is an empty bipartite subgraph.} \quad
Indeed, $[Y,R]$ is empty by Lemma  \ref{th:L1}. $[R_N,R^1_0]$ is empty by the definition of $R^1_0$.

\item {\bf $S=\{v_p,v_q\}$ and   $\Gamma_{G_m}(v_p)\cup \Gamma_{G_m}(v_q) \subseteq K'$.} \quad This is trivial from the construction.
\end{enumerate}
\end{proof}

\noindent
Therefore, we arrived at a hostile configuration, which completes the proof
 of Lemma~\ref{lem:key} in \emph{Case I}.

\subsection{Proof of Lemma~\ref{lem:key}: Case II. There exists  $1\le i \le k$ such that  $R_i$ contains edges.} \label{sec:caseII}

Then,
\begin{displaymath}
R=R_0\cup R_i\cup R_\infty \text{ and $R_i$ contains an edge $ ab$}.
\end{displaymath}
\begin{figure}[H]
    \begin{center}
    \begin{tikzpicture}[scale=1]
    \tikzstyle{vertex}=[draw,circle,fill=black,minimum size=3,inner sep=0]
    \draw  (-.2,5) ellipse (14pt and 20pt);
    \node[vertex] (vp) at (0,5.3) [label=west:{$v_p$}] {};
    \node[vertex] (vq) at (0,4.7) [label=west:{$v_q$}] {};
    \node at  (-.9, 5) {\LARGE S};
    \draw[dashed] (vp) -- (vq);
    \draw (1,0) rectangle (3,6);
    \node at (.5,3.5) {\LARGE K};
    \draw (1.5,4) rectangle (2.5,5.3);
    \node at (2,5.5) {X};
    \draw (1.5,1) rectangle (2.5,3);
    \node at (2,.8) {Z};
    \node[vertex] (vi) at (2,2) [label=west:{$v_i$}] {};
    \draw (3.4,4) rectangle (5.5,6);
    \node at (6,5) {\LARGE Y};
    \draw[line width=1mm] (1.2,5.5) -- (1.2, 1.5);
    \draw[line width=1mm, dashed] (5.1,5.3) -- (5.1,4.1);
    \draw[line width=1mm, dashed] (3.7,4.6) -- (3.7,3);
    \draw[line width=1mm, dashed] (4.5,3.1) -- (5,3.1);
    \draw (4,2.7) rectangle (5.5,3.3);
    \node at (5.8,3) {$R_\infty$};
    \draw (4,1) rectangle (5.5,2);
    \node at (5.8,1.7) {$R_i$};
    \draw[dashed] ;
    \draw (4,0) rectangle (5.5,.6);
    \node at (5.8,.3) {$R_0$};
    \node[vertex] (b) at (4.5,1.7) [label=east:{$b$}] {};
    \node[vertex] (a) at (4.2,1.2) [label=east:{$a$}] {};
    \draw[very thin] (4.8,.3) -- (5,1.5);
    \draw[very thin] (5.1,.3) -- (5.25,3);
    \draw[very thin] (4.4,.2) -- (5.1,.2);
    \draw[dashed] (b) -- (vi);
    \draw[dashed] (a) -- (vi);
    \draw (a) -- (b);
    \draw[dashed] (5,3) -- (vi);
    \draw (3.4,-0.2) rectangle (6.2,3.4);
    \node at (6.5,3) {\Large R};
    \end{tikzpicture}
    \end{center}
\caption{In $Y\cup R_\infty$ there is no edge. The solid lines are existing edges, the dashed line are missing edges The thick lines note all edges / non-edges in the subgraph. } \label{fig2}
\end{figure}
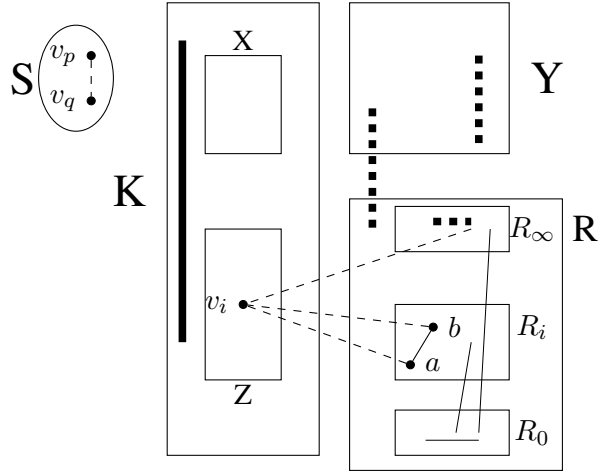

\noindent In Case II it will be more challenging  to transform  our graph $G$ into a graph $G_m$ which has a hostile configuration. We will do it in two steps.

\medskip \noindent\textbf{Step 1.}

In the first step we will "uplift" edges from $R_i$ into the bipartite graph $G[R_i,K]$
 as follows. Let $H_1,\ldots, H_t$ be an enumeration of the connected components of the subgraph $G[R_i].$ In each component $H_s$ fix  a vertex  $r_s$ and choose a spanning tree $T_s$ with root $r_s$ such that $\Gamma_{T_s}(r_s)= \Gamma_{G[R_i]}(r_s).$
Let define
\begin{equation}\label{eq:R0iR1i}
    R^1_i := \{ r_\beta : 1\le \beta \le t\} \ \text{ and }\ R^0_i :=  (R_i \setminus R^1_i).
\end{equation}

 \begin{figure}[H]
 \begin{subfigure}[t]{.5\textwidth}
    \begin{center}
    \begin{tikzpicture}
    \tikzstyle{vertex}=[draw,circle,fill=black,minimum size=3,inner sep=0]
    \draw (2,6.7) ellipse (20pt and 10pt);
    \node [vertex] (vp) at (1.5,6.7) [label=west:{$v_p$}] {};
    \node [vertex] (vq) at (2.5,6.7) [label=east:{$v_q$}] {};
    \node at (.5,6.7) {\Large S};
    \draw[dashed] (vp) -- (vq);
    \draw (1,0) rectangle (3,6);
    \node at (.5,3.5) {\LARGE K};
    \draw (1.5,4) rectangle (2.5,5.3);
    \node at (2,5.5) {X};
    \draw (1.5,1) rectangle (2.5,3);
    \node at (2,.8) {Z};
    \node[vertex] (vi) at (2,2) [label=west:{$v_i$}] {};
    \draw (3.4,4) rectangle (5.5,6);
    \node at (6,5) {\LARGE Y};
    \draw[line width=1mm] (1.2,5.5) -- (1.2, 1.5);
    \draw[line width=1mm, dashed] (5.1,5.3) -- (5.1,4.1);
    \draw[line width=1mm, dashed] (3.7,4.6) -- (3.7,3);
    \draw[line width=1mm, dashed] (4.5,3.1) -- (5,3.1);
    \draw (4,2.7) rectangle (5.5,3.3);
    \node at (5.8,3) {$R_\infty$};
    \draw (4,2) rectangle (5.5,2.5);
    \node at (5.8,2.2) {$R_i^1$};
    \draw[dashed] ;
    \draw (4,.6) rectangle (5.5,1.8);
    \node at (5.8,1.3) {$R_i^0$};
    \draw (4,-.2) rectangle (5.5,.4);
    \node at (5.8,.1) {$R_0$};
    \node[vertex] (b) at (4.5,2.2) [label=east:{$r_s$}] {};
    \node[vertex] (a) at (4.2,.8) [label=east:{$x$}] {};
    \node[vertex] (y) at (5,1.4) [label=east:{$y$}] {};
    \draw[dashed] (b) -- (vi);
    \draw[dashed] (a) -- (vi);
    \draw[dashed] (b) -- (y);
    \draw[dashed] (y) -- (vi);
    \draw (a) -- (y);
    \draw (a) -- (b);
    \draw (3.4,-0.5) rectangle (6.2,3.4);
    \node at (6.5,3) {\Large R};
    \draw[loosely dotted,very thick,->] (4.4,1.9) .. controls (3.7,2.1) ..  (3.1,1.5);
    \draw[loosely dotted,very thick,->] (4.75,1.2)  .. controls (3.8,1.1) ..  (3.8,1.6);
    \end{tikzpicture}
    \end{center}
\caption{Step 1: Two upward twist operations} \label{fig3}
\end{subfigure}
\quad
\begin{subfigure}[t]{.5\textwidth}
\begin{center}
    \begin{tikzpicture}
    \tikzstyle{vertex}=[draw,circle,fill=black,minimum size=3,inner sep=0]
    \draw (2,6.7) ellipse (20pt and 10pt);
    \node [vertex] (vp) at (1.5,6.7) [label=west:{$v_p$}] {};
    \node [vertex] (vq) at (2.5,6.7) [label=east:{$v_q$}] {};
    \node at (.5,6.7) {\Large S};
    \draw[dashed] (vp) -- (vq);
    \draw (1,0) rectangle (3,6);
    \node at (.5,3.5) {\LARGE K};
    \draw (1.5,4) rectangle (2.5,5.3);
    \node at (2,5.5) {X};
    \draw (1.5,1) rectangle (2.5,3);
    \node at (2,.8) {Z};
    \node[vertex] (vi) at (2,2) [label=west:{$v_i$}] {};
    \draw (3.4,4) rectangle (5.5,6);
    \node at (6,5) {\LARGE Y};
    \draw[line width=1mm] (1.2,5.5) -- (1.2, 1.5);
    \draw[line width=1mm, dashed] (5.1,5.3) -- (5.1,4.1);
    \draw[line width=1mm, dashed] (3.7,4.6) -- (3.7,3);
    \draw[line width=1mm, dashed] (4.5,3.1) -- (5,3.1);
    \draw (4,2.7) rectangle (5.5,3.3);
    \node at (5.8,3) {$R_\infty$};
    \draw (4,2) rectangle (5.5,2.5);
    \node at (5.8,2.2) {$R_i^1$};
    \draw[dashed] ;
    \draw (4,.6) rectangle (5.5,1.3);
    \node at (5.8,1) {$R_i^0$};
    \draw (4,-.2) rectangle (5.5,.4);
    \node at (5.8,.1) {$R_0$};
    \node[vertex] (b) at (4.5,2.2) [label=east:{$y_\ell$}] {};
    \draw[dashed] (b) -- (vi);
    \draw[rounded corners] (3.4,1.8) rectangle (6.2,3.4);
    \node at (6.7,3) {\Large R${}^\star_N$};
    \draw[rounded corners] (3.4,-.3) rectangle (6.2,1.5);
    \node at (6.7,1.2) {\Large R${}^\star_0$};
    \node[vertex] (a) at (4.2,0) [label=south:{$x_\ell$}] {};
    \node[vertex] (c) at (5.1,0) [label=south:{$z_\ell$}] {};
    \draw [dashed](a) -- (c);
    \draw (a) -- (b);
    \draw[loosely dotted,very thick,->] (4.5,1.5) .. controls (5.1,.8) ..  (4.65,0);
    \end{tikzpicture}
    \end{center}
\caption{Step 2: One downward twist operation} \label{fig4}
\end{subfigure}
\caption{Step 1 and 2}
\end{figure}
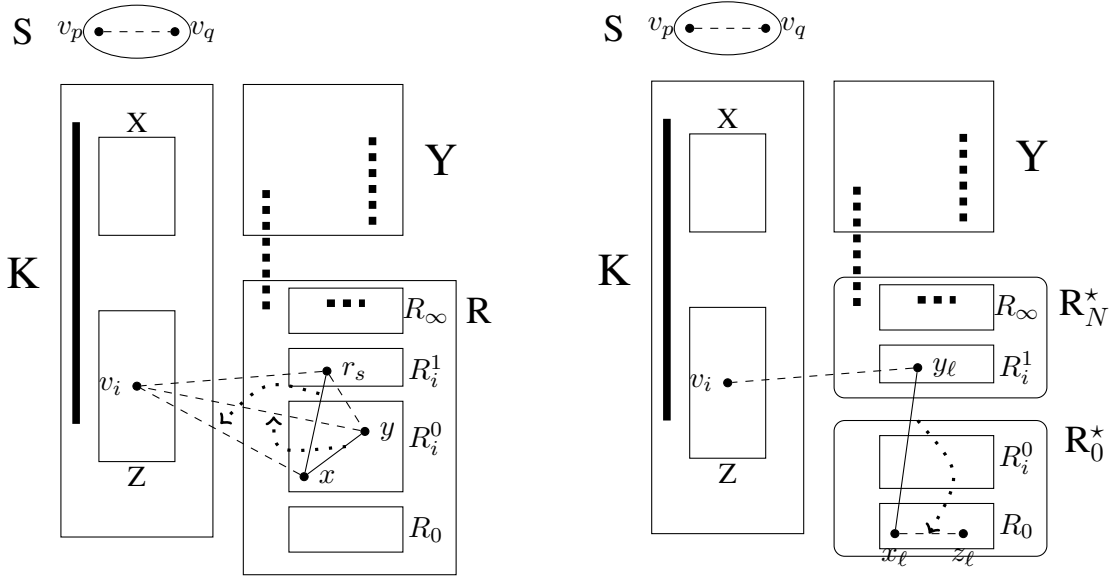
We define the \textbf{upward twist} operation of this graph as follows.
In each spanning tree $T_s$ let orient the edges outward from the root $r_s$. Then, for each oriented edge $\overrightarrow{xy}$ execute the hinge-flip operation $\hflip{y}{x}{v_i}$.

In other words,  in  each component $H_s$,  delete the edges of the spanning tree $T_s$, and connect each vertex but $r_s$ in $T_s$ to the vertex $v_i$. (See Figure \ref{fig3}.)

As a result of the upward twist operation, we obtain the  graph $G^\star$ with degree sequence $D^\star$. Since the upward twist operation was  obtained as a sequence of hinge-flip operation, $\sum D^\star=\sum D^\diamond=\sum D+2$.

\begin{lemma}\label{th:Sou-5}
\begin{equation}\label{eq:Sou-4}
\left (D^\star + 1_{-q}^{-p}\right ) \in \dds{n}{\Sigma}{c_1}{c_2}.
\end{equation}
\end{lemma}
\begin{proof}
If $\Gamma_{G^\star}(v)\ne\Gamma_G(v)$, then $v=v_i$ or $v\in R$. So we should show $c_2\le \deg_{G^\star}(v)\le c_1$ for $v\in \{v_i\}\cup R$.

If $v\in R$, then $\deg_{G^\star}(v)\le\deg_{G}(v)\le c_1$ and
$|X\setminus \Gamma_{G^\star}(v)|\le |X\setminus \Gamma_{G}(v)|\le 1$  by Lemma \ref{th:L1}, so  $\deg_{G^\star}(v)\ge |X|-1\ge c_2$. Since $\Gamma_{G^\star}(v_i)\cap K=\Gamma_{G}(v_i)\cap K=K\setminus \{v_i\}$, we have $\deg_{G^\star}(v_i)\ge |K|-1\ge c_2$.

Finally,  we should prove that  $\deg_{G^\star}(v_i),$  does not exceed $c_1.$ To show that fix  $1\le j \le k$ such that $j\ne i$ and $v_j\in \Gamma_G(v_p)\cap \Gamma_G(v_q)$.   Assume on the contrary that
\begin{displaymath}
\deg_{G^\star}(v_i)> c_1 \ge |\Gamma_{G}(v_j)|=|\Gamma_{G^\star}(v_j)|.
\end{displaymath}
Since  $R_i\setminus \Gamma_{G^*}(v_i)=\{r_1,\dots, r_t\}$, we have
$|R_i\cap  \Gamma_{G^\star}(v_i)|\le |R_i|-|\{r_1,\dots, r_t\}|   \le    |R_i|-1$. Moreover, $|K\cap  \Gamma_{G^\star}(v_i)|=|K\cap  \Gamma_{G^\star}(v_j)|=|K|-1$. Hence,
\begin{multline*}
\left( |R_i|-1 \right) + (|K|-1)  + |\Gamma_{G^\star}(v_i)\setminus (R_i\cup K)| \ge \deg_{G^\star}(v_i) > c_1\ge  \\
\ge \deg_G(v_j)=  |\Gamma_{G}(v_j)\setminus (R_i\cup K)| + |R_i|+|K|-1.
\end{multline*}
Since  $\Gamma_{G^\star}(v_i)\setminus R_i=\Gamma_{G}(v_i)\setminus R_i$,
we have
\begin{displaymath}
|\Gamma_{G}(v_i)\setminus (R_i\cup K)|> |\Gamma_{G}(v_j)\setminus (R_i\cup K)|+1.
\end{displaymath}
 Consequently, we can choose
 \begin{displaymath}
 \{d_0,d_1\}\in [\Gamma_{G}(v_i)\setminus \Gamma_{G}(v_j) \setminus (R_i\cup K)]^2.
 \end{displaymath}
Then, $d_{\alpha}\notin \Gamma_{G}(v_j)$ implies $d_{\alpha}\notin R_0\cup S$. Hence, $\{d_0,d_1\}\subset Y\cup R_{\infty}$. Since $|\Gamma_G(v_i)\cap Y|\le 1$  by Lemma \ref{lm:Ysingleton}, we can assume that $d_0\in R_\infty$. Then in $G[K\cup R]$ we have  the following pre-witness trail of length $5$:
\begin{displaymath}
v_j\cdots d_0-v_i\cdots a-b\cdots v_i.
\end{displaymath}
Contradiction, we have $\deg_{G^\star}(v_i)\le c_1$.
\end{proof}

\noindent\textbf{Step 2.}

Unfortunately, the graph $G^\star$ may contain witness trails, so we cannot simply
say: ``apply Case I for $G^\star$''.
Instead of that, in the second step we just imitate the ``downward twist'' construction of \textbf{Case I} in the graph $G^\star$. Recalling that
 $R^1_i$ and $R^0_i$  were defined in \eqref{eq:R0iR1i},
 observe that
\begin{displaymath}
E(G^\star)\setminus E(G)= G^\star[\{v_i\},R^0_i], \qquad
E(G)\setminus E(G^\star)\subset R^0_i \times (R^1_i\cup R^0_i).
\end{displaymath}
Moreover,
\begin{equation}\label{eq:R}
E(G^\star)\cap [(R^0_i\cup R^1_i) \times R^1_i]=\emptyset.
\end{equation}
We also have
\begin{equation}\label{eq:R01}
\text{$\Gamma_{G^\star}(w)\supset K$ for $w\in R^0_i$.}
\end{equation}
Write
\begin{displaymath}
R^\star_0=R_0\cup R^0_i, \qquad
R^\star_N=R_\infty\cup R^1_i.
\end{displaymath}
Observe that
\begin{equation}\label{eq:RN}
    E(G^\star)\cap [R_N^*]^2=\emptyset
\end{equation}
by Lemma \ref{lm:RiRj_free} and by $E(G^\star)\cap {\left [R^1_i\right ]}^{2}= \emptyset$.

\medskip
After this preparation,  generate a sequence of graphs $G_0, G_1,\ldots, G_m$ and their degree sequences $D_0, D_1, \ldots, D_m$ with consecutive hinge-flip operations as follows.

Let $G_0=G^\star$. Assume that $G_\ell$ is already produced. Pick
$x_\ell \ne z_\ell \in R^\star_0$ and $y_\ell \in R^\star_N$ such that $(x_\ell, y_\ell)$ is an edge and $(x_\ell, z_\ell)$ is a non-edge in $G_\ell.$ Apply the corresponding hinge-flip operation $\hflip{x_\ell}{y_\ell}{z_\ell}$ for $G_\ell$ to obtain $G_{\ell+1}$ (They are just the {downward twist} operations from \textbf{Case I}, see Figure \ref{fig4}.) We stop the construction if there are no suitable $x_\ell$, $y_\ell$ and $z_\ell$.

Observe that
\begin{equation}\label{eq:x}
\{x_\ell: \ell<m\}\subset R_0.
\end{equation}
Indeed, $(x_\ell, y_\ell)$ is an edge in $G_\ell $, and there is no edge between $R^0_i$ and $R^\star_N=R^1_i\cup R_\infty$ in $G_\ell$ by  \eqref{eq:R} and by Lemma \ref{lm:RiRj_free}.

\begin{lemma}\label{lm:marad}
$D_m+1^{-p}_{-q}\in \mD$.
\end{lemma}
\begin{proof}
Since we carried out hinge flip operations, $\sum D_m=\sum D^\star=\sum D^\diamond=\sum (D+1^{+p}_{+q})=\sum D+2$, furthermore we changed degrees  only in $R^\star_0\cup R^\star_N$.

\medskip\noindent
In $R^\star_N$ we decreased the degrees. Moreover, for each $r\in R^\star_N$, we have $|X\setminus \Gamma_{G_m}(r)|=|X\setminus \Gamma_{G^\star}(r)|\le |X\setminus \Gamma_{G}(r)|\le 1$ by Lemma \ref{th:L1}, and so $\deg_{G_m}(r)\ge |\Gamma_G(r)\cap X|-1\ge c_2$. So $c_2\le \Gamma_{G_m}(r)\le c_1$ for $r\in R_N^\star$.

\medskip\noindent
In $R^\star_0$ we increased some degrees. So, to complete the proof of the Lemma, we should show $\deg_{G_{m}}(z)\le c_1$ for each $z\in R^\star_0$.

Assume on the contrary, that $d_{G_m}(z)>c_1$ for some $z\in R^\star_0$. Since $\deg_{G^\star}(z)\le c_1$, we increased the degree of $z$  by some   hinge flip operation, i.e.,  there is $1\le \ell\le m$ such that $z_\ell=z.$ Then $x_\ell, z_\ell \in R^\star_0$ and $y_\ell \in R^\star_N$, furthermore $(x_\ell, y_\ell)$ is an edge and $(x_\ell, z_\ell)$ is a non-edge in $G_{\ell}$. By \eqref{eq:x} we know $x_\ell\in R_0$. Hence,
\begin{equation}\label{eq:EN1}
(x_\ell, y_\ell) \text{ is an edge in } G, \text{ and } (x_\ell, z_\ell) \text{ is a non-edge in } G.
\end{equation}

\medskip\noindent
Since $y_\ell\in R^*_N$, we can pick  $1\le j\le k$ such that $j\ne i$ and $v_j\in X \cap \Gamma_G(z_\ell)$. We know that the bipartite graph $G_\ell[R^\star_0,\{v_{j}\}]$ is complete by \eqref{eq:R01} and by the definition of $R_0$. Hence, we have the following inequality chain:
\begin{displaymath}
\deg_{G^\star\setminus R^\star_0}(z_\ell)+(|R^\star_0|-1)\ge \deg_{G_m}(z_\ell)>c_1\ge \deg_G(v_{j})= \deg_{G\setminus R_0}(v_{j})+|R^\star_0|.
\end{displaymath}
Moreover,
$\Gamma_{G^\star\setminus R^\star_0}(z_\ell)\setminus \Gamma_{G\setminus R^\star_0}(z_\ell)\subset \{v_i\}$,
and so $\deg_{G^\star\setminus R^\star_0}(z_\ell)\le \deg_{G\setminus R^\star_0}(z_\ell)+1.$

Putting together, we obtain
\begin{displaymath}
    \deg_{G\setminus R^\star_0}(z_\ell)>\deg_{G\setminus R_0}(v_{j}).
\end{displaymath}
Thus,  there exists a vertex $d$ such that $(z_\ell,d)$ is an edge and
$(v_{j},z_\ell)$ is a non-edge in $G$.

Then $d\notin Y\cup S$ because $(z_\ell,d)$ is an edge and $z_\ell\in R^\star_0$.
Moreover,  $d\notin K\cup R_0$ because $(v_j,d)$ is a non-edge.
Thus, the only possibility is that  $d\in R_N$.

\begin{itemize}
\item If $y\in R_i$, then let $i^\star=i$.
\item If $y\in R_\infty$, let $i^\star\ne j$ such that $(v_{i^\star},y_\ell)$ is a non-edge.
\end{itemize}
Thus $i^\star\ne j$, and
$$
v_{i^\star}\cdots y_\ell - x_\ell \cdots z_\ell-  d\cdots v_{j}
$$
is a prohibited pre-witness trail of length $5$ in $G[K\cup  R]$.

The obtained contradiction proves that $d_{G_m}(z)\le c_1$ for all $z\in R^\star_0$.
\end{proof}

We carried out the construction, and we verified that the degree sequence $D_m$ can be obtained as $D_m'+1^{+p}_{+q}$ for some $D'_m\in \mD$.
Finally, we will find a hostile configuration for $G_m$. Let
\begin{displaymath}
    R_K=\{r\in R^\star_0:
    \text{there is $G_m$-edge between $r$ and $R^\star_N$}\}=\Gamma_{G_m}(R_N^*)\cap R^*_0.
\end{displaymath}

\begin{lemma}\label{lm:observe}\label{eq:RK}
        $R_K\subset R_0$ and $[R_K,R^\star_0]\subset E(G_m)$.
\end{lemma}

\begin{proof}
    Since
    \begin{equation}\label{eq:sl1}
    E(G_m)\cap [R^0_i,R_i^1]\subset E(G^*)\cap [R^0_i,R_i^1]=\emptyset
    \end{equation}
    by  \eqref{eq:R},
    and
    \begin{equation}\label{eq:sl2}
    E(G_m)\cap [R^0_i,R_\infty]\subset E(G)\cap [R_i,R_\infty]=\emptyset
    \end{equation}
    by Lemma \ref{lm:RiRj_free}, we have
    \begin{equation}\label{eq:sl3}
    E(G_m)\cap [R^0_i,R_N^*]=(E(G_m)\cap [R^0_i,R_i^1])\cup (E(G_m)\cap [R^0_i,R_\infty])=\emptyset.
    \end{equation}
    Hence, $R_K\cap R^0_i=\emptyset$, and so $R_K\subset R_0$.

To show $[R_K,R^\star_0]\subset E(G_m)$
assume on the contrary that $x\in R_K$ and $z\in R^\star_0$ such that $(x,z)\notin E(G_m)$.
Since $x\in R_K$, there is $y\in R^*_N$ such that $(x,y)$ is an edge in $G_m$. So we have a  further  downward twist operation  $\hflip{x}{y}{z}$ in $G_m$. Contradiction.
\end{proof}

\begin{figure}[H]
\begin{center}
    \begin{tikzpicture}[scale=1]
    \tikzstyle{vertex}=[draw,circle,fill=black,minimum size=3,inner sep=0]
    \draw  (-.2,5) ellipse (14pt and 20pt);
    \node[vertex] (vp) at (0,5.3) [label=west:{$v_p$}] {};
    \node[vertex] (vq) at (0,4.7) [label=west:{$v_q$}] {};
    \node at  (-.9, 5) {\LARGE S};
    \draw[dashed] (vp) -- (vq);
    \draw (1,0) rectangle (2.8,6);
    \node at (1.7,-.3) {\Large K};
    \draw (1.5,4) rectangle (2.5,5.3);
    \node at (2,5.5) {X};
    \draw (1.5,1) rectangle (2.5,3);
    \node at (2,.8) {Z};
    \node[vertex] (vi) at (2,2) [label=west:{$v_i$}] {};
    \draw (4,5) rectangle (5.5,6);
    \node at (5.8,5.5) {Y};
    \draw[line width=1mm] (1.2,5.5) -- (1.2, 1.5);
    \draw[line width=1mm, dashed] (3.6,5.3) -- (3.6,3.2);
    \draw (4,3.8) rectangle (5.5,4.8);
    \node at (5.8,4.2) {$R_\infty$};
    \draw (4,2.7) rectangle (5.5,3.7);
    \node at (5.8,3.2) {$R_i^1$};
    \draw[rounded corners] (3.4,2.5) rectangle (6.2,6.2);
    \node at (6.7,4.5) {\Large Y'};
    \draw[dashed] ;
    \draw (4,1.1) rectangle (5.5,2);
    \node at (5.8,1.6) {$R_i^0$};
    \draw (4,.1) rectangle (5.5,.9);
    \node at (6.3,.4) {$R_0\setminus R_K$};
    \draw[rounded corners] (3.4,-.1) rectangle (7.1,2.25);
    \node at (7.5,1.2) {\Large R'};
    (4.65,0);
    \draw (4,-1.2) rectangle (5.5,-.4);
    \node at (5.9,-.9) {$R_K$};
    \draw[dashed,thick,rounded corners] (.8,6.2) -- (.8,-1.5) -- (6.3,-1.5) -- (6.3,-.3) -- (2.9,-.3) -- (2.9,6.2)-- (.8,6.2);
    \node at (0.3,2.5) {\LARGE K'};
    \end{tikzpicture}
    \end{center}
\caption{The hostile configuration after \textbf{Step 2}} \label{fig5}
\end{figure}
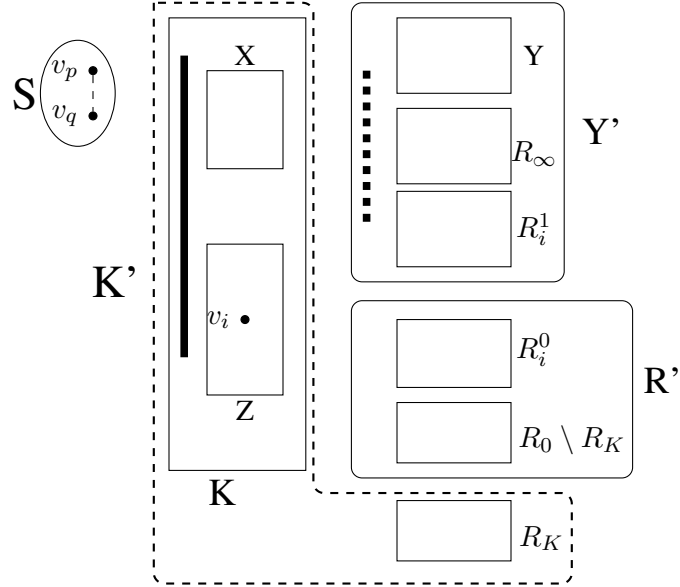

\begin{lemma}\label{lm:dream-2}
$G_m$ has a hostile configuration.
\end{lemma}
\begin{proof}
We have the  set $S=\{v_p,v_q\}$, and let
    \begin{displaymath}
    K'=K\cup R_K,\qquad
     Y'=Y\cup R^*_N, \qquad
     R'=R^*_0\setminus R_K.
    \end{displaymath}
We should check that conditions \ref{df:hostile}.(a)-(e)  hold for the partition     $(S,K',Y',R')$.

\begin{enumerate}[(a)]
\item {\bf $G_m[K']$  is a complete subgraph.} \quad
    Indeed, $K$ is a clique by Lemma \ref{th:L1}(iii). $R_K$ is a clique by Lemma \ref{eq:RK}. $[R_K,K]$ is complete because $R_K\subset R_0$ by Lemma \ref{eq:RK}, and $[R_0,K]$ is complete by the definition of $R_0$.
\item  {\bf $G_,[K';R']$ is a complete bipartite subgraph.} \quad
    Indeed, $[K,R^*_0]$ is complete by \eqref{eq:R01}.  Moreover, $R'\subset R_0^*$ and  $[R_K,R_0^*]\subset E(G_m)$ by Lemma \ref{eq:RK}.
\item  {\bf $G_m[Y']$ is an empty subgraph.} \quad  Indeed $[Y,Y\cup R]$ is empty in $G$
by Lemma  \ref{th:L1}(ii,iv) , and $R^*_N$ is empty by \eqref{eq:RN}.
\item {\bf $G_m[Y';R']$ is an empty bipartite subgraph.} \quad
    Indeed, $[Y,R]$ is empty by Lemma \ref{th:L1}. $[R_N^*,R']$ is empty by the definition of $R_K$.
\item {\bf  $S=\{v_p,v_q\}$ and  $\Gamma_{G_m}(v_p)\cup \Gamma_{G_m }(v_q) \subseteq K'$.} \quad
    Indeed, $\Gamma_{G_m}(v_p)\cup \Gamma_{G_m }(v_q)= \Gamma_{G}(v_p)\cup \Gamma_{G}(v_q) \subseteq K\subseteq K'$.
\end{enumerate}
Hence, the partition really witnesses a hostile configuration.
\end{proof}
This completes the proof of Lemma~\ref{lem:key} \qed

\subsection{An application}\label{sec:apply}

Theorem \ref{tm:main} provides a new, general, method to verify the $P$-stability of certain regions. As a consequence we obtain the following strengthening of the statement of condition (P3):
\begin{theorem}\label{th:SF+}
For any constant $ \varepsilon>0 $, the simple region $\mathbb D[{\varphi}_{\ds \! {G\!S\!+}}]$ defined by the inequality
$$
\varphi_{\ds \!{G\!S\!+}} \equiv \left (n \geq \frac{1}{2\ep^2}, 2\le c_2 \text{ and } 3 \le c_1 \le \sqrt{(1-\varepsilon)\Sigma}\right).
$$
is fully graphic and hence $P$-stable.
\end{theorem}

\begin{proof}
In \cite[Theorem 3.4]{fully} it was proved that if
\[
D\in\mD[2\leq c_2, 3\leq c_1\leq \sqrt{(1-\varepsilon)\Sigma}]
\]
is not fully graphic, then
\begin{equation}\label{eq:main-gs2}
n<  \frac{1}{8(1-\sqrt{1-\varepsilon})^2} \leq \frac{1}{2 \varepsilon^2}
\end{equation}
So for all $\varepsilon$ there exists an $n_\varepsilon$ that for all $n> n_\varepsilon$ the region $\mD[{\varphi}_{\ds \! {G\!S\!+}}]$ is fully graphic.

\end{proof}

\section{Almost all not fully graphic simple regions are also not $P$-stable}\label{sec:almost}

We begin by recalling a useful characterization of graphicality from~\cite{fully} and use it to prove Lemma~\ref{lem:fully}.
For the simple degree sequence region $\dds{n}{\Sigma}{c_1}{c_2}$  we define  the degree sequence $\lb(n,\Sigma,c_1,c_2)$ as follows:
\begin{equation}\label{eq:LEG}
\lb(n,\Sigma,c_1,c_2)= \underbrace{c_1,\ldots,c_1}_{\fal}, a, \underbrace{c_2, \ldots, c_2}_{(n-1)-\fal},
\end{equation}
where
\begin{equation}\label{eq:alpha}
\alpha= \frac{\Sigma - n\cdot c_2}{c_1-c_2}
\quad\text{and}\quad
a = c_2 + (c_1-c_2)\cdot \{\alpha\}.
\end{equation}

\noindent The following simple statement characterizes the fully graphic simple degree sequence regions in terms of its associated $\lb$ sequence:
\begin{lemma}[{\cite[Theorem 2.5]{fully}}]\label{lem:LEG}
The simple region $\dds{n}{\Sigma}{c_1}{c_2}$ is fully graphic if and only if the degree sequence $\lb(n,\Sigma,c_1,c_2)$ \ is graphic. \qed
\end{lemma}

\begin{proof}[Proof of Lemma~\ref{lem:fully}]

Assume that a simple degree sequence region $\mD(n,\Sigma,c_1,c_2)$ is not fully graphic.
Then Lemma~\ref{lem:LEG} implies that, in particular, the
sequence $\lb(n,\Sigma,c_1,c_2)$ is not graphic.

Tripathi and Vijay showed in \cite{TV03} that for this particular type of sequence non-graphicality implies $\alpha \geq c_2$ and that the Erdős–Gallai inequality fails for either $k=\fal$ or $k = \ceal$. That is,
\begin{align*}
c_1 k &> k(k-1) + \min(c_2,k)(n-k-1) + \min(a, k) &\geq& k(k-1)+c_2(n-k)  \mbox{ for } k = \fal \mbox{ or}\\
c_1 k &\geq a + c_1 (k-1)> k(k-1) + \min(c_2,k)(n-k) &\geq& k(k-1) + c_2(n-k) \mbox{ for } k = \ceal.
\end{align*}
Therefore, we have for $k=\fal$ or $k=\ceal$ that
\begin{equation}\label{eq:k=a3}
    c_1 k > k(k-1) + c_2(n-k),
\end{equation}
or, equivalently, that
\begin{equation}\label{eq:alpha2}
0 > k^2 - (c_2 + c_1 + 1)k + nc_2.
\end{equation}
A simple algebraic manipulation shows that the discriminant of the right-hand side, $ (c_2+c_1+1)^2-4nc_2$, is equal to $Q$ from \eqref{eq:fully-iff}. Since the discriminant has to be positive, this implies \eqref{eq:fully-iff}. Furthermore, $\fal$ or $\ceal$ has to lie between the two roots of the equation
\[
k^2 - (c_2 + c_1 + 1)k + nc_2 = 0,
\]
which implies \eqref{eq:fully-iff2} by the identity \eqref{eq:alpha}.
\end{proof}

The proof strategy of Theorem~\ref{thm:sigmabound} is rather straightforward. There is a known construction that yields a degree sequence with an exponential boundary: the degree sequence of the Tyskevich product of a half-graph and a bipartite graph. Our aim is then to show that there is a such a sequence in $D(n,\Sigma,c_2,c_1)$ whenever $\Sigma$ satisfies \eqref{eq:sigma_abs_bound}.

\subsection{Half-graphs and split degree sequences}

The unique realization of the following degree sequence is called the \textbf{half-graph}. (The name was originally coined by Paul Erd\H{o}s and A. Hajnal.)
$$
\mathbf{h}_{r/2}=(1,2,\dots,r/2,r/2,\dots, r-1).
$$
Its relevance for us comes from the following estimate from~\cite{JMS92}
\begin{equation}\label{eq:fully-iff3}
\pppp {\mathbf{h}_{r/2}}= \Theta\left(3^{r/2}\right)
\end{equation}
We will consider a graph $G$ where $V(G)$ can be partitioned into three disjoint subsets $X \uplus Y \uplus R$ such that:
\begin{itemize}
    \item $G[X]$ is a complete subgraph,
    \item $G[Y]$ is an empty subgraph,
    \item $G[X,R]$ is a complete bipartite graph,
    \item $G[Y,R]$ is an empty bipartite graph,
    \item The degrees sequence of $G[R]$ is $(\mathbf{h}_{r/2}).$
\end{itemize}
According to the proof of  \cite[Theorem 8.2]{JMS92},  the degree sequence of this graph has boundary quotient at least $\pppp {\mathbf{h}_{r/2}}$.

Let us fix the size of $R$ to an even value \(r=|R|\). For any \(c_2 \leq x \leq c_1-r+1\) we can consider a partition with \(x=|X|\) and \(n-r-x=y=|Y|\). We are going to determine, what values of $\Sigma$ can a graph have that is constructed in the above manner, and with the given sizes for $X,Y,R$. Everything in the graph is fixed except the bipartite portion between the sets $X$ and $Y$, so
\[\Sigma = r^2/2 + 2xr+ x(x-1)+2|E(X,Y)|.\] For that subgraph the only constraints are that the nodes in $Y$ need to have degree $\geq c_2$ while nodes in $X$ must have degree $\leq c_1-x-r+1$.

\begin{lemma}
    For any $e$ satisfying
\begin{align}
    c_2 (n-x-r) \leq e \leq x(c_1-x-r+1)\label{eq:a_cons}
\end{align}
there is a bipartite graph $G'(X,Y)$ with $e$ edges that satisfies the degree constraints described above.
\end{lemma}

\begin{proof}
    Pick a bi-degree sequence that takes only 2 adjacent values on the $X$ side as well as on the $Y$ side, with the sum being equal to $e$ on both sides. This sequence clearly satisfies the lower and upper constraints, and it is easy to verify that there is a bipartite graph with this given bi-degree sequence.
\end{proof}

\begin{definition}\label{eq:Ix}
Let $I^x$ denote the interval given by
\begin{align*}
I^x &= [I^x_0, I^x_1] \; \mbox{ where}\\
I^x_0 &:=r^2/2 + x(x+2r-1) + 2c_2(n-x-r) \; \mbox{ and}\\
I^x_1 &:= r^2/2 + x(x+2r-1) + 2x(c_1 - x-r+1).
\end{align*}
\end{definition}

\begin{corollary}\label{cor:sigmaint} If $\Sigma \in I^x$ for some $c_2 \leq x \leq c_1-r+1$ then $D(n,\Sigma,c_2,c_1)$ is not $P$-stable.
\end{corollary}

From now on our goal is to show that $\cup I^x$ covers the interval described in \eqref{eq:sigma_abs_bound}. Upon closer inspection, it is evident that both $I^x_0$ and $I^x_1$ are monotone increasing in $x$. The key to showing that their union covers a large interval, then, is to show that they overlap for consecutive values of $x$.

\begin{claim}
    $I^x \cap I^{x+1} \neq \emptyset$ if and only if $I^{x+1}_0 \leq I^{x}_1$, that is when
    \[ (x+1)(x+2r)+2c_2(n-x-1-r) \leq x(x+2r-1) + 2x(c_1-x-r-+1),\]
    or, equivalently,
    \begin{equation}\label{eq:newxbound}
        x^2-x(c_1+c_2-r)+r + c_2(n-1-r) \leq 0
    \end{equation}
\end{claim}
\noindent Let $\amin$ and $\amax$ denote the smallest and largest integers satisfying \eqref{eq:newxbound} and let
\begin{align*}
    \Smin &= I^{\amin}_0 = \frac{r^2}{2} +  \amin \cdot (\amin +2r-1)+ 2 c_2 (n-\amin - r) \\
\Smax &= I^{\amax}_1 = \frac{r^2}{2} +  \amax \cdot (\amax +2r-1)+ 2 \amax (c_1-\amax - r+1).
\end{align*}

\begin{corollary}
The simple region $D(n,\Sigma,c_2,c_1)$ is not $P$-stable for
\[ \Smin \leq \Sigma \leq \Smax \]
\end{corollary}

\subsection{Bounding the interval}

\begin{proof}[Proof of Theorem~\ref{thm:sigmabound}] All we need to establish that the interval defined by \eqref{eq:sigma_abs_bound} is contained in $[\Smin,\Smax]$. Or rather, we need to determine a value of $\ep$ that allows us to establish this, and then we need to prove that this $\ep \leq \frac{3(r+3)}{\beta(c_1-c_2)}$.

Using that $\amin$ and $\amax$ satisfy \eqref{eq:newxbound} we can further bound the endpoints of the $[\Smin,$ $\Smax]$ interval:
\begin{align}\nonumber
\Smin = I^{\amin}_0 &\leq I^{\amin}_0 + \amin(c_1+c_2-r)-\amin^2-c_2(n-1-r) \\ \nonumber
    &=\frac{r^2}{2}+ \amin(c_1-c_2+r-1)+c_2(n-r-1) -r\\\nonumber
    &\leq \amin(c_1-c_2) +c_2 n + r\left(\amin - c_2 + \frac{r}{2} \right) \\ \label{eq:sminlower}  &\leq (\amin+r)(c_1-c_2) + nc_2 \\
\nonumber
\Smax =I^{\amax}_1 &\geq I^{\amax}_1 + \amax^2-\amax(c_1+c_2-r) + c_2(n-1-r) \\ \nonumber
    &=\frac{r^2}{2} + \amax(c_1-c_2+r+1)+c_2(n-1-r)+r \\ \nonumber
    &\geq \amax(c_1 - c_2) + c_2 n + r\left(\amax - c_2 + \frac{r}{2} \right)\\ \label{eq:smaxlower}
    &\geq \amax(c_1 - c_2) + c_2 n
\end{align}
Thus, to show $\Smin \leq \Sigma \leq \Smax$ it suffices to prove
\[ \amin +r \leq \frac{\Sigma - n c_2}{(c_1-c_2) } \leq \amax \]
From \eqref{eq:sigma_abs_bound} we know that
\[ (c_1+c_2+1) - (1-\varepsilon)\sqrt{Q} \leq \frac{\Sigma-nc_2}{(c_1-c_2)/2} \leq (c_1+c_2+1)+(1-\ep)\sqrt{Q},\] hence we will be done if we choose $\ep$ such that
\begin{align*}
2r+2\amin&\leq(c_1+c_2+1)-(1-\ep)\sqrt{Q} \;\mbox{ and}\\
2\amax &\geq (c_1+c_2+1)+(1-\ep)\sqrt{Q}
\end{align*}
We know that \(\amax\) is the largest integer solution of \eqref{eq:newxbound} hence
\[\amax \geq \frac{(c_2+c_1-r)+\sqrt{(c_2+c_1-r)^2-4c_2(n-r-1)-4r}}{2} - 1\] so
\begin{align*}
     2\amax - (c_1+c_2+1) &\geq \sqrt{(c_2+c_1-r)^2-4(c_2(n-r-1)+r)} - (r+3)\\
     &= \sqrt{ (c_1-c_2 -r)^2 - 4c_2(n-c_1-1)-4r} - (r+3)
\end{align*}
Similarly, $\amin$ is the smallest integer solution of \eqref{eq:newxbound}, so
\[\amin \leq \frac{(c_2+c_1-r)-\sqrt{(c_2+c_1-r)^2-4c_2(n-r-1)-4r}}{2} + 1\] so
\begin{align*}
     2r + 2\amin - (c_1+c_2+1) &\leq -\sqrt{(c_2+c_1-r)^2-4(c_2(n-r-1)+r)} + (r+1)\\
     &= -\sqrt{ (c_1-c_2 -r)^2 - 4c_2(n-c_1-1)-4r} + (r+1)
\end{align*}
To simplify notation, it will be convenient to introduce:
\[ Q(r) := (c_1-c_2-r)^2 - 4c_2(n-1-c_1)+4r.\]
Now we simply set
\begin{equation}\label{eq:epsilon}
\varepsilon = 1 -\frac{\sqrt{Q(r)}-r-3}{\sqrt{Q}}.
\end{equation}
This is equivalent to
\[ (1-\ep)\sqrt{Q} = \sqrt{Q(r)} -(r+3).\]
Putting everything together, for $\Sigma$ satisfying \eqref{eq:sigma_abs_bound} we can combine all the above steps and get
\begin{align*}
\frac{\Smin - c_2 n}{(c_1-c_2)/2} - (c_1+c_2+1)&\leq 2r+\amin -(c_1+c_2+1)\\ &\leq r+1 - \sqrt{Q(r)}\\ &\leq
-(1-\varepsilon)\sqrt{Q}  \\ &\leq\frac{\Sigma - n c_2}{(c_1-c_2)/2}-(c_1+c_2+1) \\ &\leq (1-\varepsilon)\cdot \sqrt{Q} \\ &= \sqrt{Q(r)}-(r+3)\\ &\leq 2\amax - (c_1+c_2+1) \leq \frac{\Smax - n c_2}{(c_1-c_2)/2}-(c_1+c_2+1)
\end{align*}
and hence $\Smin \leq \Sigma \leq \Smax$

The only task left is to bound $\varepsilon$. To that end, let us calculate
\begin{align*}\sqrt{Q}-(\sqrt{Q(r)}-r-3) &=r+3+ \frac{Q-Q(r)}{\sqrt{Q}+\sqrt{Q(r)}}\\
&= r+3+\frac{(r+1)(2c_1-2c_2+1-r)-4r}{\sqrt{Q}+\sqrt{Q(r)}}\\
&\leq (r+3) \cdot \left(\frac{\sqrt{Q}+2c_1-2c_2+1-r}{\sqrt{Q}}\right) \\
&\leq (r+3)\cdot \left(\frac{3c_1 - 3c_2 +2 -r}{\sqrt{Q}}\right)\\
&\leq 3(r+3) \cdot \frac{(c_1-c_2)}{\sqrt{Q}},
\end{align*}
where we used $\sqrt{Q} \leq c_1 - c_2 +1$ along with some rather crude estimates. Recall that our assumption was $(1-\beta)(c_1-c_2+1)^2 \geq 4c_2(n-1-c_1)$, equivalently $Q > \beta (c_1-c_2+1)^2$. Using this assumption on $Q$ we see that
\[ \varepsilon = \frac{\sqrt{Q}-\sqrt{Q(r)}+r+3}{\sqrt{Q}} \leq 3(r+3) \frac{c_1-c_2}{Q} \leq 3(r+3) \frac{c_1-c_2}{\beta (c_1-c_2+1)^2} \leq \frac{3(r+3)}{\beta(c_1-c_2)}\]
as claimed.
\end{proof}

\section{Future directions}
Infinite $P$-stable degree sequence sets are known to exist for both bipartite and directed graphs. Moreover, for these classes, the associated switch Markov chains are known to be rapidly mixing (see [3]). Nevertheless, our understanding of such infinite sets remains considerably less complete than in the simple (undirected) case, although the known examples display notable structural analogies. Consequently, we intend to extend our investigation to the bipartite and directed settings, where we expect to find analogous relationships between $P$-stable sets and “fully graphic” sets.

\bibliographystyle{plain}

\end{document}